\documentclass[10pt]{article}

\usepackage[utf8x]{inputenc}
\usepackage[english]{babel}
\usepackage[T1]{fontenc}
\usepackage{amsmath,amscd}
\usepackage{amsfonts}
\usepackage{amssymb}
\usepackage{amsthm}
\usepackage{mathrsfs}
\usepackage[all,cmtip]{xy}
\usepackage{textcomp}
\usepackage{qtree}
\usepackage{url}
\usepackage{amsopn,latexsym}
\usepackage[pdftex]{graphicx}
\usepackage{fullpage}
\usepackage{amsmath,graphicx}

\newtheorem{teo}{Theorem}[section]
\newtheorem{lemma}{Lemma}[section]
\newtheorem{rem}{Remark}[section]
\newtheorem{prop}{Proposition}[section]
\newtheorem{cor}{Corollary}[section]
\newtheorem{es}{Example}[section]
\newtheorem{defin}{Definition}[section]

\newcommand{\Z}{{\mathbb{Z}}}

\newcommand{\R}{{\mathbb{R}}}

\newcommand{\N}{{\mathbb{N}}}

\newcommand{\G}{{\mathcal{G}}}

\newcommand{\MS}{{\mathcal{S}}}

\newcommand{\A}{{\mathcal{A}}}

\newcommand{\Sc}{{\mathcal{S}}}

\newcommand{\F}{{\mathcal{F}}}
\newcommand{\T}{{\mathcal{T}}}

\newcommand{\emme}{{\mathcal M}}
\newcommand{\mcap}{{\widehat{\mathcal M}}(\A^\perp)}

\title{Permutonestohedra}
\author{Giovanni Gaiffi\footnote{Dipartimento di Matematica, Universit\`a di Pisa}}

\begin{document}

\maketitle

\begin{abstract}
There are  several real spherical models associated with a  root arrangement, depending on the choice of a {\em building set}.
 The connected components of these  models are manifolds with corners  which can be glued together to obtain the corresponding real De Concini-Procesi models.
In this paper, starting from  any  root  system \(\Phi\) with finite Coxeter group \(W\) and any \(W\)-invariant building set,  we describe an explicit  realization of the real spherical model as a union of polytopes (nestohedra) that  lie inside  the   chambers of the arrangement. The main point  of this realization is that   the  convex hull of these nestohedra is a larger polytope, a {\em permutonestohedron}, equipped with an action of \(W\) or also, depending on the building set,  of \(Aut(\Phi)\).  The permutonestohedra are   natural   generalizations of  Kapranov's permutoassociahedra.

\end{abstract}

\section{Introduction} 

Let \(V\) be  an euclidean vector space of dimension \(n\), let \(\Phi\subset V\) be a root system which spans \(V\) and has  finite Coxeter group \(W\), and let \(\G\) be a \(W\)-invariant building set associated with  \(\Phi\) (see Section \ref{sec:ricordare} for a definition of building set).

 The main goal  of this paper is to provide an explicit linear realization  of the {\em permutonestohedron} \(P_\G(\Phi)\),  a polytope  whose face poset was  introduced in \cite{gaimrn}; this polytope is  linked with   the geometry of the  wonderful model of the  arrangement and is equipped with an action of \(W\) or also, depending on the building set,  of \(Aut(\Phi)\).
 
  Let us  briefly describe the framework from which the permutonestohedra arise. 

Given  a   building set \(\G\)  one can construct as in \cite{DCP1} a real and a  complex  De Concini-Procesi model (resp. \(Y_\G (\R)\) and   \(Y_\G\)), or their  `spherical' version,   a real model with corners \(CY_\G\) (see \cite{gaimrn0}).  

These  models play a relevant role in several   fields of mathematical research:  subspace and toric  arrangements,  toric varieties (see for instance \cite{DCP3}, \cite{feichtneryuz}, \cite{rains}),   
moduli spaces of curves and       configuration spaces (see for instance \cite{etihenkamrai}, \cite{LTV}), box splines and  index theory (see the exposition in  \cite{DCP4}),  discrete  geometry    (see \cite{feichtner} for further references).

From the point of view of discrete geometry and combinatorics the  spherical model  \(CY_\G\) is a particularly interesting object. 
In fact   \(CY_\G\) has as many  connected components as the number of chambers and 
these connected components  are non linear realizations of some  polytopes that belong to the family of nestohedra. The nestohedra  have been defined and studied in   \cite{postnikov}, \cite{postnikoreinewilli}, \cite{zelevinski},  and successively in several other papers, due to the interest of nested sets complexes in  combinatorics    and geometry (see for instance  \cite{feichtnersturmfels} for applications to  tropical geometry and \cite{cavazzanimoci} for applications to  the study of  Dahmen-Micchelli modules).

 Moreover we  remark that  a suitable gluing of the connected components of  \(CY_\G\) produces the model \(Y_\G (\R)\), which is therefore presented as a CW complex.
If  \(\G\) is \(W\)-invariant,  \(Y_\G (\R)\) (as well as  \(Y_\G\)) inherits an action of \(W\) which gives rise to geometric  representations in cohomology.

The permutonestohedron \(P_\G(\Phi)\) arises in the middle of  this rich algebraic and geometric picture.
Its  name comes from the following remarks:
\begin{itemize}
\item   some of its facets  lie inside the chambers  of the arrangement;  they are nestohedra, and their union   is a \(W\)-invariant linear realization of \(CY_\G\);  their convex hull is the full permutonestohedron;   
\item if \(\Phi=A_n\) and \(\G\) is the minimal building set associated to \(\Phi\), the  permutonestohedron is a realization of  Kapranov's permutoassociahedron (see \cite{Kapranov}). 

\end{itemize}
We remark that a  non linear  realization  in \(V\) of \(CY_\G\) (and of a convex body whose face poset is the same as the face poset of \(P_\G(\Phi)\))   was provided in \cite{gaimrn}.

As we explain in Section \ref{sec:stasheff}, the linear realization of \(CY_\G\) that we describe in this paper   is inspired by (and generalizes in a way) Stasheff and Shnider's construction of the associahedron in the Appendix B of \cite{stasheffshnider}.

The main point in our choice of   the  defining hyperplanes  is that the every  nestohedron in \(CY_\G\)  lies inside a chamber of the arrangement; furthermore,  each of these  hyperplanes  is  invariant with respect to the action of a parabolic subgroup.  These are the reasons why, in  the global construction, when we   consider the convex hull of  all  the nestohedra that lie in the chambers,  the extra facets of  the permutonestohedron \(P_{\G}(\Phi)\) appear. These turn out to be combinatorially equivalent to the product of a nestohedron  with some smaller permutonestohedra (see Theorem \ref {combinatorialiso}).

We notice that when \(\Phi=A_1^n\), as the building set varies,  the nestohedra we obtain   inside  every orthant are  all  the nestohedra in the `interval simplex-permutohedron' (see \cite{petric2}). 
In particular  in the case of the  associahedron our construction coincides with the one in \cite{stasheffshnider}, while  for the other  nestohedra it is analogue to   the constructions   in   \cite{dosenpetric}  and  \cite{devadossforcey}, \cite{carrdevadoss2},   that  are remarkable     generalizations of Stasheff and Shnider's construction.  

In general, once \(\Phi\) is fixed,   the nestohedron  which lies inside a chamber depends  on \(\G\).  For instance, the minimal building set associated to \(\Phi\) is the building set made by  the {\em irreducible} subspaces, i.e., the subspaces spanned by the  irreducible root subsystems, while the maximal building set is made by all the subspaces spanned by some of the roots. Now, if \(\Phi\) is \(A_{n}\), \(B_n\) or \(C_n\) and \(\G\) is the minimal building set,  the nestohedron is a \(n-1\) dimensional Stasheff's associahedron;  if \(\Phi\) is \(D_{n}\)  and \(\G\) is the minimal building set, the nestohedron is a graph associahedron of type \(D\). For any \(n\)-dimensional root system, if \(\G\) is the maximal building set  the nestohedron is a \(n-1\) dimensional permutohedron (so in this case our permutonestohedron could  also be called `permutopermutohedron').
 When   the root system is  of type \(A_n\), \(B_n\), \(C_n\), \(D_n\) the full poset of invariant  building sets has  been described  in \cite{GaiffiServenti2}.


As mentioned before, we obtain the permutonestohedron \(P_\G(\Phi)\) by taking the convex hull of the nestohedra that  lie in the chambers and whose union realizes \(CY_\G\).  From this point of view the  construction of the permutonestohedron is  inspired by   Reiner and Ziegler construction of {\em Coxeter associahedra}  in \cite{reinerziegler},  since these    are obtained as the convex hull of some   Stasheff's associahedra that  lie inside the chambers.  
Anyway we remark that  only when \(\Phi=A_n\) and \(\G\) is the minimal building set the permutonestohedron  \(P_\G(\Phi)\) is combinatorially equivalent to a Coxeter associahedron (i.e., in   the case of Kapranov's permutoassociahedron). 
For instance   in the \(B_n\) case  the Coxeter associahedron is the convex hull of some   \(n-2\) dimensional Stasheff's associahedra, while our  minimal permutonestohedron  is the convex hull  of some  \(n-1\) dimensional Stasheff's associahedra; furthermore, non minimal permutonestohedra are convex hulls of different nestohedra.

 \begin{figure}[h]
 \center
\includegraphics[scale=0.4]{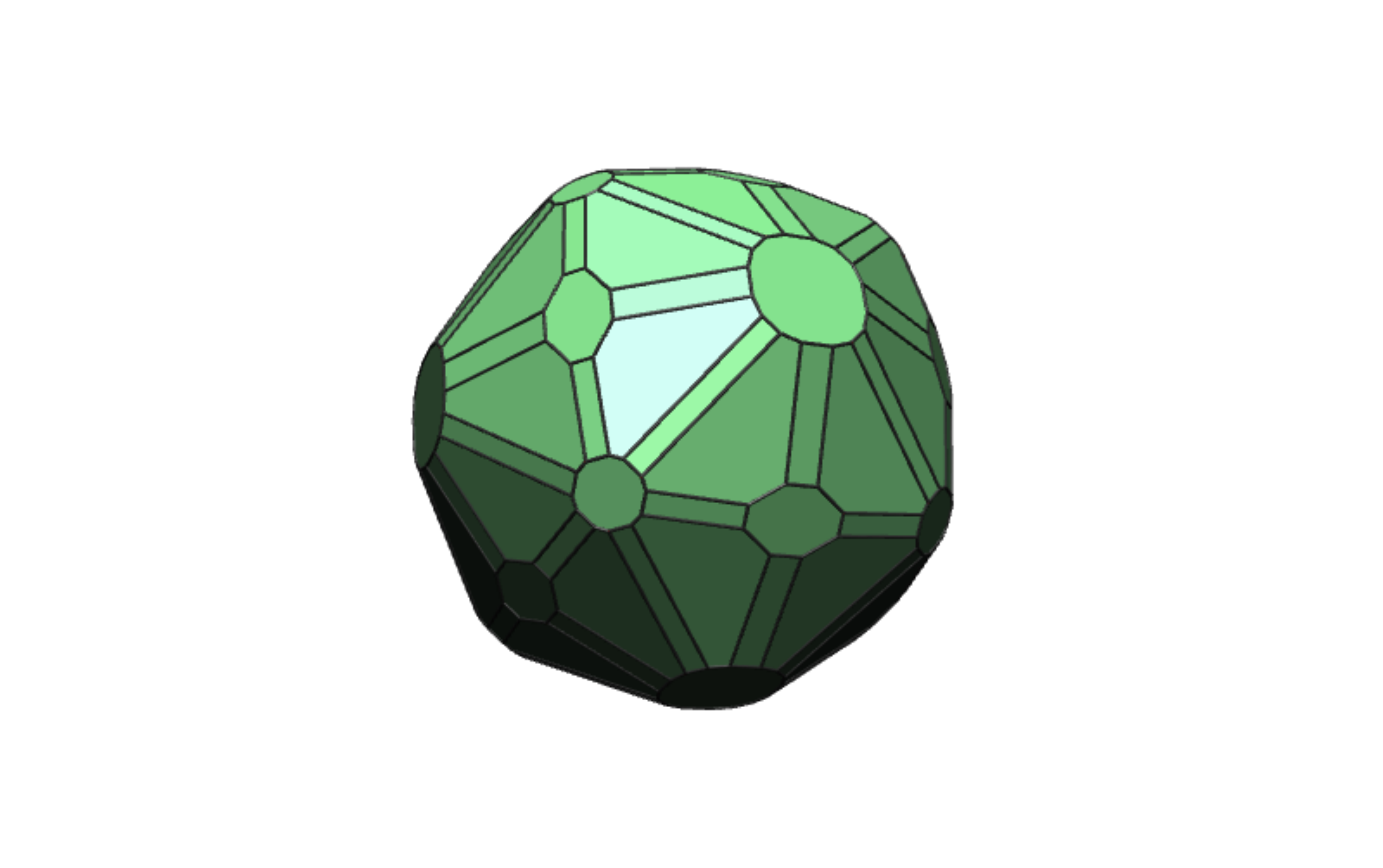}
\caption{The maximal permutonestohedron of type \(B_3\): inside each chamber of the arrangement there is a nestohedron (in this case a hexagon, i.e. a  two dimensional permutohedron), and the permutonestohedron is the convex hull of all these nestohedra.}
\label{prima}
\end{figure}

We will also focus on the group of isometries of the  permutonestohedron \(P_\G(\Phi)\): it  contains \(W\)  or even, depending on the `symmetry' of \(\G\), the group \(Aut (\Phi)\), and we will   describe  some  conditions which are sufficient  to ensure that it is equal to \(Aut (\Phi)\).   We will also point out that a subposet of the minimal permutonestohedron of type \(A_{n-1}\) is equipped with an   `extended' action of \(S_{n+1}\).

Finally, here it is a short  outline of this paper. In Section \ref{sec:ricordare} we briefly  recall the basic properties of  building sets, nested sets and spherical models,  while Section \ref{sec:permutonestohedra} contains the definition of permutonestohedra and the statement of main theorems, Theorems \ref{teo:nestoedro} and \ref{teo:verticipermutoedro},  whose proofs  can be found   respectively in Section \ref{sec:nestohedron} and Section \ref{sec:permutodimostrazioni}.

Section \ref{sec:poset} is devoted to a presentation  of the face poset of the permutonestohedron \(P_\G(\Phi)\); this  points out the action of \(W\),  since  the faces of  \(P_\G(\Phi)\) are  indicized by pairs of the form (coset of \(W\),  labelled nested set). Furthermore,  Theorem \ref{combinatorialiso}   shows   that every facet that does not lie inside a chamber is combinatorially equivalent to  the product of a nestohedron  with some smaller permutonestohedra. We also show (Corollary \ref{cor:nonsemplice}) that   \(P_\G(\Phi)\)  is a simple polytope if and only if \(\G\) is the maximal building set associated to \(\Phi\).

The full group of isometries that  leave \(P_\G(\Phi)\) invariant is explored in Section \ref{sec:autgroup}: if \(\G\) is invariant with respect to \(Aut (\Phi)\) this group contains  \(Aut (\Phi)\),   and we describe  some  natural conditions that  are sufficient  to ensure that it is equal to \(Aut (\Phi)\).

In Section \ref{sec:extendedaction} we show how the  well know `extended' \(S_{n+1}\) action on the minimal De Concini-Procesi model of type \(A_{n-1}\)  can be lifted to  a subposet of the minimal permutonestohedron of type \(A_{n-1}\), providing geometrical realizations of the representations \(Ind_G^{S_{n+1}} Id\), where \(G\) is any subgroup of the cyclic group of order \(n+1\).

Finally in   Section \ref{sec:examples}   we show some  examples and pictures of  permutonestohedra and, as an example of face counting, we compute the \(f\)-vectors of the minimal and of the maximal permutonestohedron  of type \(A_n\).

All the  pictures of three-dimensional polytopes (like the one in Figure \ref{prima}) have been realized using the mathematical software package {\it polymake} (\cite{polymake}).


\bigskip

\section{Building sets, nested sets, nestohedra and spherical models}
\label{sec:ricordare}
 Let us  recall from \cite{DCP1}, \cite{DCP2} the  definitions of nested set and building set of subspaces; more precisely, we  specialize these definitions to the case  we are interested in, i.e. the case when we deal with  a central hyperplane   arrangement.

Let  $\A$ be a central line    arrangement in an Euclidean space \(V\).
We denote   by ${\mathcal C}^{}_{\A}$ the closure  under the sum of $\A^{}$ and 
by $\A^{\perp}$ the hyperplane arrangement 
\[\A^\perp=\{A^\perp \:|\: A\in \A\}\]

\begin{defin}
\label{geometricbuilding}
The  collection of subspaces $\G \subset {\mathcal C}^{}_{\A}$ is called  {\em building set associated to \(\A\)} if \(\A\subset \G\) and  every element $C$ of ${\mathcal C}^{}_{\A}$ is the direct sum 
 $C= G_1\oplus G_2\oplus \ldots\oplus G_k$ of  the maximal elements 
$G_1,G_2,\ldots,G_k$  of
$\G^{}$ contained in $C$. This is called the \(\G\)-decomposition of \(C\).
\end{defin}
Given a hyperplane  arrangement  \(\A\), there are several building sets associated to it. Among these there always are a maximum and a minimum (with respect to inclusion). The maximum is ${\mathcal C}^{}_{\A}$, the minimum is the building set of {\em irreducibles}. In the case of a root  arrangement the building set of irreducibles is the set of  subspaces spanned by  the irreducible root subsystems of the given root system (see \cite{YuzBasi}).




 \begin{defin}
Let $\G$ be a building set associated to \(\A\). 
A subset $\MS\subset \G$ is called 
($\G$-) nested, if
 given a subset $\{U_1,\ldots ,U_h \}$ (with \(h>1\)) of pairwise non comparable 
	elements in $\MS$, then $U_1\oplus \cdots \oplus U_h\notin \G$.
\end{defin}

After De Concini and Procesi's paper \cite{DCP1},    nested sets and building sets appeared in the literature, connected with several combinatorial problems.  
One can see  for instance  \cite{feichtnerkozlovincidence},  where   building sets and nested sets  were defined in the more general context of meet-semilattices, and  \cite{delucchinested} where the connection with Dowling lattices was investigated.  Other purely combinatorial definitions were used   to  give rise to  the polytopes that  were named {\em nestohedra} in \cite{postnikoreinewilli}  and turned out to play a relevant role in discrete mathematics and tropical geometry. Presentations  of nestohedra can be found for instance  in    \cite{postnikov},
\cite{postnikoreinewilli}, \cite{zelevinski}, \cite{feichtnersturmfels}, \cite{carrdevadoss}, \cite{carrdevadoss2}, \cite{petric2}.

Here we recall, for the convenience of the reader,  the combinatorial definitions of building sets and nested sets as they appear in \cite{postnikov},
\cite{postnikoreinewilli}. One can  refer to Section 2 of \cite{petric2} for a short comparison among   various definitions and notations in the literature. 

\begin{defin}
\label{combinatorialbuilding}
A building set \({\mathcal B}\) is a set of subsets of \(\{1,2,...,n\}\), such that:
\begin{itemize}
\item[a)] If \(A,B \in {\mathcal B}\) have nonempty intersection, then \(A\cup B\in {\mathcal B}\).
\item[b)]The set \(\{i\}\) belongs to \({\mathcal B}\) for every \(i\in \{1,2,...,n\}\).
\end{itemize}

\end{defin}

\begin{defin}
\label{combinatorialnested}
A subset  \({\mathcal S}\) of a building set  \({\mathcal B}\) is a nested set if and only if the following three conditions hold: 
\begin{itemize}
\item[a)] For any \(I,J\in {\mathcal S}\) we have that either \(I\subset J\) or \(J\subset I\) or \(I\cap J=\emptyset\).
\item[b)] Given elements  \(\{J_1,...,J_k\}\) (\(k\geq 2\)) of \({\mathcal S}\) pairwise not comparable with respect to inclusion, their union is not in \({\mathcal B}\). 
\item[c)] \({\mathcal S}\) contains all the sets of \({\mathcal B}\) that  are maximal with respect to inclusion. 
\end{itemize}

\end{defin}

The nested set complex  \({\mathcal N}({\mathcal B})\) is the poset of all  the nested sets of \({\mathcal B}\) ordered by inclusion. 
A  nestohedron \(P_{\mathcal B}\) is a polytope whose face poset, ordered by reverse inclusion, is isomorphic to the nested set complex \({\mathcal N}({\mathcal B})\).

Let us now consider a  `geometric' building set \(\G\) of subspaces associated with a root  arrangement,  according to the Definition \ref{geometricbuilding}, and let us suppose that \(V\in \G\).  
\begin{defin}
\label{defingfund} We will denote by \(\G_{fund}\) the building set made by the subspaces in \(\G\) that are spanned by  some subset of the set of simple roots.
\end{defin}
Now \(\G_{fund}\)  gives rise to a building set in the sense of the Definition \ref{combinatorialbuilding} in the following way: we  associate to a subspace \(A\in \G_{fund}\)  the set of indices of  the simple roots contained in it. 

Having established this correspondence,  the  nested sets in the sense of  De Concini and Procesi  are  nested sets in the sense of the Definition \ref{combinatorialnested} provided that they  contain   \(V\), the maximal subspace in \(\G_{fund}\). These nested sets  form a nested set complex denoted by  \({\mathcal N}(\G_{fund})\).
The  nestohedra that will appear in our construction, as facets of permutonestohedra,   are the nestohedra \(P_ {\G_{fund}}\) associated with  the nested set complexes \({\mathcal N}(\G_{fund})\).



Now let    $\A$ be, as above,  a central line    arrangement in an Euclidean space \(V\), and let \(\emme(\A^{\perp})\) be the complement in \(V\)  to \(\A^\perp\).
In \cite{gaimrn0} the following  
compactifications of $\mcap=\emme(\A^{\perp})/\R ^+$ were defined, in the spirit of De Concini-Procesi construction of wonderful models.

Let us denote  by $S(V)$ the $n-1$-th dimensional unit sphere in $V$, and, for 
every subspace $A\subset V$, let $S(A)=A\cap S(V)$. Let \(\G\) be a building set associated to \(\A\), and    let us consider 
the compact manifold 
\begin{displaymath}
	K=S(V)\times\prod _{A\in \G}S(A^{})
\end{displaymath}
There is an open embedding 
\(
\phi:\: 	\mcap\longrightarrow K
\)
which  is obtained as a composition of the 
section $s:\:\mcap \mapsto \emme(\A^\perp)$  
\begin{displaymath}
	s([p])=\frac{p}{|p|}\in S(V)\cap \emme(\A^\perp)
\end{displaymath} 
with the map
\begin{displaymath}
 	\emme(\A^\perp)\mapsto S(V)\times\prod _{A\in \G}S(A ^{})
\end{displaymath}
that is well defined  on each factor. 
\begin{defin} 
We denote by $CY_{\G}$ the closure in $K$ of $\phi(\mcap)$.
\end{defin}
It turns out (see \cite{gaimrn0}) that 
$CY_{\G}$ is a smooth manifold with corners. Its (not connected) boundary components  are in correspondence with the elements of the building set \(\G\), and the intersection of some boundary components is nonempty if and only if  these components correspond to a nested set.
We notice that  \(CY_\G\) has as many  connected components as the number of chambers of \(\emme(\A^\perp)\). 

In \cite{gaimrn}  these connected components were  realized inside the chambers, as the complements of   a suitable set of  tubular neighbouroods of the subspaces in \(\G^\perp\), giving rise to  some non linear realizations of the nestohedra \(P_ {\G_{fund}}\).

\section{The construction of permutonestohedra}
\label{sec:permutonestohedra}
\subsection{Selected  hyperplanes}
\label{secgeneral}
The goal of this section is  to give  the equations of the hyperplanes that will be used in the definition of the permutonestohedra.

Let \(V\) be as before  an euclidean vector space of dimension \(n\)  with scalar product \((  \ , \ )\).
Let us consider  a root  system \(\Phi\) in \(V\) with finite Coxeter group \(W\), and a basis of {\em simple roots}  \(\Delta=\{\alpha_1,...,\alpha_n\}\) for \(\Phi\). If \(\Phi\) is irreducible,  we consider  in the open fundamental chamber \(Ch(\Delta)\) the vector 
\[\delta=\frac{1}{2}\sum_{\alpha \in \Phi^+}\alpha =\sum_{i=1}^{n}\omega_i\]
where \(\Phi^+\) is the set of positive roots and the simple weights \(\omega_i\) are defined by 
\[\frac{2(\alpha_i, \omega_i)}{(\alpha_i,\alpha_i)}=\delta_{ij}\]  (see  \cite{bourbaki4}, \cite{humphreyscoxeter} as  general references on root systems and  finite Coxeter  groups).

If \(\Phi\) is not irreducible and splits into the irreducible subsystems \(\Phi_1,\Phi_2, \ldots, \Phi_s\) we put \(\delta=\sum \delta_{\Phi_i}\) where, for every \(i\), \(\delta_{\Phi_i}\) is the semisum of all the positive roots of \(\Phi_i\).

Let us denote by \(\mathcal{C}_\Phi\)  the building set of all the subspaces that can be generated as the span of some of the roots in \(\Phi\) and by  \(\F_\Phi\)  the building set  of all the irreducible subspaces in \(\mathcal{C}_\Phi\). As we recalled in Section \ref{sec:ricordare},   \(\F_\Phi\) is made by  all the subspaces   that  are spanned by  the irreducible root subsystems of \(\Phi\).

 Let  \(\G\) be a building set associated to the root system \(\Phi\) 
 which contains \(V\) and is   \(W\)-invariant; when the root system is  of type \(A_n\), \(B_n\), \(C_n\), \(D_n\) these building sets have been classified in \cite{GaiffiServenti2}.


 Given \(A\in \G_{fund}-\{V\}\), we  will denote by \(W_A\) the parabolic subgroup generated by the reflections \(s_\alpha\) with \(\alpha\in \Phi\cap A\). We denote by \(\delta_A\)  the  orthogonal projection of \(\delta\) to \(A^\perp\); we have 
\[\delta_A=\frac{1}{|W_A|}\sum_{\sigma \in W_A} \sigma (\delta)\]
and we  also write
\(\delta_A=\delta-\pi_A\)
where \(\pi_A\) belongs to \(A\).

\begin{rem}
\label{pia}
If \(A\cap \Phi\) is an irreducible root subsystem,  we notice  that \(\pi_A\) is the semisum of all its  positive  roots. If \(A\cap \Phi\) is not irreducible, and splits into the irreducible subsystems \(\Phi_1,...,\Phi_s\), then \(\pi_A=\sum \pi_{\Phi_i}\)  where \(\pi_{\Phi_i}\) is the  semisum of all the   positive  roots of   \(\Phi_i\).  

Let \(I\) be the set of indices made by the  \(i\in\{1,2,...,n\}\) such that \(\alpha_i\in A\) and let \(J \) be the set of indices, in the complement of \(I\), such that \(j\in J\) iff \((\alpha_j, \alpha_i)<0\) for some \(i\in I\). Then 
 \[\pi_A=\sum_{i\ s.t. \ \alpha_i\in A}\omega_i - \sum_{j\in J}c_j\omega_j\]
 with all \(c_j>0\).
Therefore \[\delta_A=\sum_{s\in \{1,2,...,n\}-I}b_s\omega_s\]  with all \(b_s\geq 1\).
\end{rem}

\begin{rem} We put \(\delta_V=\pi_V=\delta\). Despite appearances,  this notation will not be confusing.
\end{rem}

Let us consider two subspaces  \(B\subset A\) in \( \G_{fund}\)  of  dimension \(j<i\) respectively and write \(\pi_A\) and \(\pi_B\) as   non negative  linear combinations of the simple roots.  We denote 
by  \(a\)  the maximum coefficient of \(\pi_A\) and by \(b\) the minimum coefficient of \(\pi_B\) and put  \(R^A_B=\frac{a}{b}\).

We then define \(R^i_{j}\) as the maximum among  all the \(R^A_B\) with \(A,B\) as above.
 
\begin{defin}
\label{defsuitablelist}
A  list of positive real numbers \(\epsilon_1<\epsilon_2<\ldots < \epsilon_{n-1}< \epsilon_n=a\) will be {\em suitable} if \(\epsilon_i>2 R^i_{i-1}\epsilon_{i-1}\) for every  \(i=2,....,n\).
\end{defin}

We are now in position to define  a set of selected hyperplanes that depend on the choice of  a suitable list  \(\epsilon_1<\epsilon_2<\ldots < \epsilon_{n-1}< \epsilon_n=a\) and are  indicized by  the elements of \(\mathcal{C}_{\Phi_{fund}}\). These hyperplanes,  together with their images via the \(W\) action, will be the defining hyperplanes of the permutonestohedron \(P_\G(\Phi)\).

The motivation for  this definition of suitable list will be pointed out  in Section \ref{sec:stasheff}. 

We start from:
\begin{defin}Let us denote by \(H_V\) the hyperplane 
\[H_V=\{x\in V \: | \: (x, \delta)=a\}\]
and by \(HS_V\) the closed half-space that contains the origin and whose boundary is \(H_V\).
\end{defin}

For every \(i=1,2,...,n\) we  call \(v_i\)  the intersection  \(H_V\cap <\omega_i>\);    all the vectors \(v_i\) lie on the  hyperplane \(H_V\) and  their convex hull, as it is well known, is a (\(n-1\))- simplex.

\begin{defin}
For every \(A\in \G_{fund}-\{V\}\),  we  define  the hyperplane
\(H_A\) as 
\[H_A=\{x\in V \: | \: (x, \delta_A)=a-\epsilon_{dim \ A}\}\]
and we denote by \(HS_A\) the closed half-space that contains the origin and whose boundary is \(H_A\).

\end{defin}
 Now we have to define the hyperplanes indicized by the elements of \( \mathcal{C}_{\Phi_{fund}}-\G_{fund}\):
\begin{defin}
\label{definHB}
Let \(B\) be a subspace in  \( \mathcal{C}_{\Phi_{fund}}-\G_{fund}\), i.e. a subspace which does not belong to \(\G_{fund}\) and is generated by some of the simple roots,  and let  \(B=A_1\oplus A_2\oplus \cdots \oplus A_k\) where \(A_1,A_2,...,A_k\) is  its \(\G_{fund}\)-decomposition. Then we put 
\[\overline{H}_B=\{x\in V \: | \: (x,\delta_B)= a-\epsilon_{dim \ A_1}-\epsilon_{dim \ A_2}- \cdots - \epsilon_{dim \ A_k}\}\] 
where \(\delta_B=\delta-\sum_{i=1}^k\pi_{A_i}\). We denote  by \(\overline{HS}_B\) the closed half-space that contains the origin and whose boundary is \(\overline{H}_B\).
\end{defin}
 

 

\subsection{Definition of the permutonestohedra}
\label{definizioneenunciati}
This section starts with the   definition of the permutonestohedron \(P_\G(\Phi)\) as the intersection of closed half-spaces. We then state two theorems that  show that the permutonestohedron is a convex hull of nestohedra and explicitly determine its vertices.

\begin{defin}
\label{defpermuto}
The  {\em permutonestohedron} \(P_\G(\Phi)\)  is the polytope given by the intersection of the half-spaces \(\sigma HS_{ A}\)  and of  the half-spaces \(\sigma\overline{HS}_{ B}\), for all \(\sigma \in W\), \(A\in \G_{fund}\) and \(B\in \mathcal{C}_{\Phi_{fund}}-\G_{fund}\).
\end{defin}

\begin{rem}
We are in fact defining infinite permutonestohedra, depending on the choice of a suitable list \(\epsilon_1<\epsilon_2<\ldots < \epsilon_{n-1}< \epsilon_n=a\), so we should write \(P_\G(\Phi,\epsilon_1,\epsilon_2,\ldots,\epsilon_n)\) instead than \(P_\G(\Phi)\). We  use the shorter notation \(P_\G(\Phi)\) since in this paper the dependence on the suitable list will be relevant only when we will deal with the automorphism group in Section \ref{sec:autgroup}, where  we will take care of avoiding   ambiguities.
\end{rem}

Now, given  a subset \(\T\) of \(\G_{fund}\) containing \(V\) and of cardinality \(n\),  with the property  that the vectors \(\{\delta_A \: | \: A\in \T\}\) 
are linearly independent,  we denote by 
\(v_\T\) the vector defined  by 
\[ v_\T= \bigcap_{A\in \T}H_A\]

As we will observe in Proposition \ref{basenested},  every    maximal nested   set \(\MS\) of \(\G_{fund}\)  has the above mentioned property.

The following theorems will be proven in  Section \ref{sec:nestohedron} and Section \ref{sec:permutodimostrazioni}.

\begin{teo}
\label{teo:nestoedro}
The intersection of \(H_V\) with all the half-spaces \(HS_{ A}\), for  \(A\in \G_{fund}-\{V\}\),  is a realization of the nestohedron \(P_{\G_{fund}}\) which lies in the open fundamental chamber.
Its  vertices are the vectors 
\(v_\MS\) where  \(\MS\) ranges among the  maximal nested sets in \(\G_{fund}\).
\end{teo}

\begin{teo}
\label{teo:verticipermutoedro}
The  permutonestohedron \(P_\G(\Phi)\) coincides with  the convex hull of all the  nestohedra 
\(\sigma P_{\G_{fund}}\) (\(\sigma \in W\)) that  lie inside the  open chambers. Its vertices are \(\sigma v_\MS\), where \(\sigma \in W\) and   \(\MS\) ranges among the  maximal nested sets in \(\G_{fund}\).
\end{teo}



\subsection{Motivations for the choice of suitable lists}
\label{sec:stasheff}

In the proofs of Theorems \ref{teo:nestoedro}  and \ref{teo:verticipermutoedro} the properties of suitable lists will play a crucial role. This is why we devote this section to suitable lists:  we  will prove  a  key lemma and we will show how  these lists generalize  Stasheff and Shnider's choice of parameters  in their construction of the associahedron in  \cite{stasheffshnider}.

The following lemma introduces an important inequality that is satisfied by   suitable lists and is tied to  the combinatorics of root systems.

\begin{lemma}
\label{lemmaepsilon}
Let \(\epsilon_1<\epsilon_2<\ldots < \epsilon_{n-1}< \epsilon_n=a\) be a suitable list of positive numbers. Let \(B\) be a subspace in \(\G_{fund}\) that can be expressed   as a sum  of some  subspaces \(B_1,B_2,...,B_r\) in \(\G_{fund}\)  (\(r>1\)), and let this sum be  non-redundant, i.e. if we remove anyone of the subspaces \(B_i\) the sum of the others is strictly included in \(B\).
Then we have 
\[\epsilon_{dim \ B}>\sum_{i=1}^r R^{dim \ B}_{dim \ B_i}\epsilon_{dim \ B_i}\]

\end{lemma}
\begin{proof}
Let \(m=dim \ B\). We notice that by definition of the numbers \(R^i_{j}\) we have:
\[\sum_{i=1}^rR^{m}_{dim \ B_i}\pi_{B_i}\gtrdot \pi_B\]
where \(\alpha \gtrdot \beta\) means that the difference \(\alpha-\beta \) can be expressed as a non negative  linear combination of the simple roots.
Let \(m-k\) be the maximum among the  dimensions of the \(B_i\)'s. Then \(r\leq k+1\) because of the non-redundancy of the \(B_i\)'s: anyone of the \(B_i\)'s contains at least a  simple root which is not contained in the others.
Now \(\epsilon_m>2R^m_{m-1}\epsilon_{m-1}\) by definition, and then recursively we obtain
\[\epsilon_m> 2^kR^m_{m-1}R^{m-1}_{m-2}\cdots R^{m-k+1}_{m-k}\epsilon_{m-k}\]
Since \(R^m_{m-1}R^{m-1}_{m-2}\cdots R^{m-k+1}_{m-k}\geq R^{m}_{m-k}\) and \(2^k\geq k+1\) (when \(k\geq 1\)), this implies: 
\[\epsilon_m> 2^kR^m_{m-k}\epsilon_{m-k}\geq (k+1)R^m_{m-k}\epsilon_{m-k}\geq r R^m_{m-k}\epsilon_{m-k} \]
The final inequality 
\[rR^m_{m-k}\epsilon_{m-k}\geq \sum_{i=1}^r R^{m}_{dim \ B_i}\epsilon_{dim \ B_i}\]
is now straightforward.
\end{proof}

Depending on \(\Phi\) and on the building set \(\G\), there may be  lists of numbers  \(\epsilon_i\) that are not suitable but can be used to construct permutonestohedra, since they  ensure that  the claim of the above lemma is true. 

More generally,  we could construct permutonestohedra using the  following  {\em suitable functions}:
first,  for every  set of subspaces  \(B, B_1,B_2,...,B_r\) as in the statement of the Lemma \ref{lemmaepsilon},  let us choose  numbers \(R^B_{B_i}\) such that 
\[\sum_{i=1}^rR^{B}_{B_i}\pi_{B_i}\gtrdot \pi_B\]
\begin{defin}
A function \(\epsilon\: : \: \G\mapsto \R^+\) is suitable if,  for every  set of subspaces  \(B, B_1,B_2,...,B_r\) in \(\G_{fund}\) as above and for every \(w\in W\),  it satisfies  
\[\epsilon(wB)>\sum_{i=1}^r R^{B}_{B_i}\epsilon(wB_i)\]
\end{defin}
This is the essential property we need in our construction of permutonestohedra. As  it is shown by the Lemma \ref{lemmaepsilon}, given a suitable list of numbers one obtains a suitable function by putting \(\epsilon (C)=\epsilon_{dim \ C}\) for every \(C\in \G\). 

We chose to use suitable lists to make our construction more concrete  and to obtain more symmetry: since the associated suitable function depends only on the dimension of the subspaces, if \(\G\) is \(Aut(\Phi)\)-invariant the automorphism group of the permutonestohedron  includes \(Aut (\Phi)\), as we will show  in Section \ref{sec:autgroup}.
Anyway, the definition of the hyperplanes \(H_A\), \(\overline{H}_B\), and the  proofs of Theorem \ref{teo:nestoedro} and Theorem \ref{teo:verticipermutoedro} in the next sections could be repeated almost verbatim using a suitable function instead of a suitable list. 

A further interesting aspect  of suitable lists and suitable functions is illustrated  by the example of the root system  \(A_1^n\),  that corresponds to the boolean arrangement.

In this case our definition of suitable list consists in   the condition \(\epsilon_i>2\epsilon_{i-1}\) for every \(i>1\).  As for suitable functions, in their definition  we can  choose all the numbers \(R^B_C\) equal to 1. Now let us number from 1 to \(n\) the positive roots of \(A_1^n\).  Then let us denote by  \(\G'\) the \(W=\Z_2^n\)-invariant  building set  made by the subspaces that  are spanned by a set of positive roots whose associated numbers are an {\em interval} in \([1,...,n]\). For this \(\G'\) we have \(\G'=\G'_{fund}\) and the corresponding nestohedron  \(P_{\G'_{fund}}\) that we obtain in the fundamental chamber  is a Stasheff's associahedron. In fact in  this case one immediately checks that our   suitable functions are the same  as  the  suitable functions used by  Stasheff and Shnider in their construction of the associahedron in Appendix B of \cite{stasheffshnider}.  

In  Section 9 of \cite{dosenpetric} Do\v{s}en and Petri\'c describe  a   generalization of  Stasheff and Shnider's construction to all the  nestohedra that  are in the `interval simplex-permutohedron'. These nestohedra are exactly all the nestohedra obtained as \(\G\) varies among the \(\Z_2^n\)-invariant building sets associated to \(A_1^n\) and   our suitable functions for this root system  are compatible with the ones described in \cite{dosenpetric}.

\section{The nestohedron \(P_{\G_{fund}}\)}
\label{sec:nestohedron}

This section is devoted to the proof of Theorem \ref{teo:nestoedro}. We will give a self-contained proof that the hyperplanes  \(H_A\) for \(A\in \G_{fund}\)  define a realization of the  nestohedron  \(P_{\G_{fund}}\); as we have  remarked in Section \ref{sec:stasheff}, our construction is  a generalization of   Stasheff and Shnider's construction of the associahedron in Appendix B of \cite{stasheffshnider}.  We  notice  that other  constructions of   nestohedra  can be found for instance in    \cite{carrdevadoss}, \cite{carrdevadoss2}, \cite{dosenpetric}, \cite{petric2}, \cite{postnikov}, \cite{postnikoreinewilli},   \cite{zelevinski}.   

Our  choice of  the  hyperplanes ensures  that the resulting nestohedron lies inside a chamber of the arrangement. Furthermore  for every  \(A\in \G_{fund}\) the hyperplane \(H_A\) is fixed by the action of \(W_A\).  Thanks to these properties  when  we pass to the global construction, and  consider the convex hull of  all  the nestohedra which lie in the chambers,  the extra facets of  the permutonestohedron \(P_{\G}(\Phi)\) appear.

\begin{prop}
\label{prop:nonappartiene}
Let us consider a subset \(\T\) of \(\G_{fund}\) containing \(V\) and of cardinality \(n\) such that the vectors \(\{\delta_A \: | \: A\in \T\}\)  are linearly independent:  if \(\T\) is not nested the vector \(v_\T\) does not belong  to \(P_{\G_{fund}}\).

\end{prop}
\begin{proof}
If \(v_\T\) doesn't belong to the open fundamental chamber, it does not belong to  \(P_{\G_{fund}}\).
In fact let us  write  a vector  \(x\in P_{\G_{fund}}\) in terms of the basis \(\omega_i\), 
\(x=\sum b_i\omega_i\);     since \(x\) belongs to \( HS_{<\alpha_i>}\) for every simple root \(\alpha_i\in \Delta\),  we have \((x,\pi_{<\alpha_i>})\geq \epsilon_1>0\), which implies \(b_i>0 \) for every \(i\). \footnote{This has the following interpretation:  after truncating the starting simplex (the one generated in \(H_V\) by the vectors \(v_i\), see Section \ref{definizioneenunciati}) by the hyperplanes \( H_{<\alpha_i>}\) we have a polytope which lies in the fundamental chamber and which, after other truncations, will become \(P_{\G_{fund}}\).}

Let us then consider  the case when \(v_\T\) lies in the open fundamental chamber.

Let \(B\) be a  subspace in \(\G_{fund}\)  that can be expressed as a sum of some (more than 1) subspaces in \(\T\). 
Such  \(B\) exists since \(\T\) is not nested.
Now  let \(B=B_1+\cdots +B_r\) with \(r>1\), \(\{B_1,B_2,...,B_r\}\subset \T\),   be a non-redundant sum (see the statement of Lemma \ref{lemmaepsilon}).


First we show that \(B\notin \T\). Since  \(v_\T\) is  in the open fundamental chamber,  it can be written as \(v_\T=\sum_{i=1,...n}b_i\omega_i\) with all the coefficients \(b_i> 0\).
Then     if \(B\in \T\) we have
\((v_\T, \pi_B)=\epsilon_{dim \ B}\) and \((v_\T, \pi_{B_i})=\epsilon_{dim \ B_i}\).
According to the definition of the numbers \(R^i_j\) we have that  
\[\sum_{i=1,...,r} R^{dim \ B}_{dim \ B_i}\pi_{B_i}\gtrdot \pi_B\]

Then we deduce  that   \((v_\T, \sum_{i=1,...,r} R^{dim \ B}_{dim \ B_i} \pi_{B_i})\geq (v_\T,\pi_B)\), which is a contradiction since \(\epsilon_{dim \ B}> \sum_{i=1,...,r} R^{dim \ B}_{dim \ B_i}\epsilon_{dim \ B_i}\) by Lemma \ref{lemmaepsilon}.

So  we can assume  \(B\notin \T\). We will  show that \(v_\T\notin HS_B\).


We notice that 
\[HS_B\cap H_V=\{x\in H_V \: | \: (x, \delta_B)\leq a-\epsilon_{dim \ B}\}=\]
\[=\{x\in H_V \: | \: (x, \delta -\pi_B)\leq a-\epsilon_{dim \ B}\}=\{x\in H_V \: | \: (x, \pi_B)\geq \epsilon_{dim \ B}\}\]

Let us then check if \(v_\T\) belongs to \(HS_B\).
As we observed before we have 
\[(v_\T, \pi_B)\leq (v_\T,  \sum_{j=1,...,r}R^{dim \ B}_{dim \ B_i}\pi_{B_j})=\sum_{j=1,...,r} R^{dim \ B}_{dim \ B_i}\epsilon_{dim \ B_j}\]
 Now Lemma \ref{lemmaepsilon} implies 
\((v_\T, \pi_B)<   \epsilon_{dim \ B}\),  which  proves that \(v_\T\) does not belong to   \(HS_B\)
(so it does not belong  to \(P_{\G_{fund}}\)).

\end{proof}

\begin{prop} 
\label{basenested}
 If   \(\MS\)  is a nested set of \(\G_{fund}\)  the vectors \(\{\delta_A \: | \: A\in \MS\}\)  are linearly independent. If \(\MS\) is a maximal nested set, the vectors \(\{\delta_A \: | \: A\in \MS\}\) are a basis of \(V\) and  the vectors \(v_\MS\) lie inside the fundamental chamber.

\end{prop}
\begin{proof}
It is sufficient to prove the linear independence  for maximal nested sets, since every nested set can be completed to a maximal nested set.

Let then \(\MS\) be a maximal nested set (therefore it contains \(V\)).
As we have already observed, the vectors \(\{\delta_A \: | \: A\in \MS\}\) are linearly independent if and only if the vectors  \(\{\pi_A \: | \: A\in \MS\}\) are linearly independent.  If these vectors are  not linearly independent,  since  for every \(A\) the vector \(\pi_A\) belongs to \(A\), we can   find a minimal  \(C \in \MS\) such that the vectors \(\{\pi_D \: | \: D\in \MS, D\subseteq C\}\) are not linearly independent. Therefore  we have, by nestedness of \(\MS\) and by minimality of \(C\),  a  relation of the form
\[\pi_C=\sum_{D\in \MS, D\subsetneq C}\gamma_D\pi_D\]
This is a contradiction, since \(C\)  contains  a simple root \(\alpha_i\) which is not contained in any \(D\in \MS, D\subsetneq C\) \footnote{Otherwise \(C\) would be equal to the sum of the subspaces \(D\), and this would contradict the nestedness of \(\MS\).}: when \(\pi_C\)   is written in terms of the simple roots \(\alpha_j\), its  \(i-th\) coefficient is >0,  while when we write \(\pi_D\) (\(D\in \MS, D\subsetneq C\)) in terms of the simple roots  the \(i-th\) coefficient is equal to 0.

This proves the linear independence, and therefore  \(v_\MS\) is  well defined.

Let us now prove that \(v_\MS\) lies in the fundamental chamber. Let us consider the  graph   associated to \(\MS\). This graph is a tree and coincides with the Hasse diagram of the poset induced by the  inclusion relation. Therefore it can be considered as  an oriented  rooted tree: the root is \(V\) and the orientation goes from the root to the leaves, that are the minimal subspaces of \(\MS\). 
We observe that we can partition the set of   vertices of the tree into {\em levels}:  level 0 is made by the leaves, and in general,   level \(k\)  is made by the vertices \(v\) such that the maximal length of an oriented  path that connects \(v\) to a leaf is \(k\).

Let \(v_\MS=\sum_{i=1}^nb_i\omega_i\).  Since \(\MS\) is maximal, it contains   at least a subspace of dimension 1. In particular, all the minimal subspaces, i.e. the leaves of the graph,  have dimension 1. Let  \(<\alpha_{i_1}>,...,<\alpha_{i_r}>\) be  the subspaces of dimension 1 in \(\MS\).
From the relation \((v_\MS,\pi_{<\alpha_{i_j}>})=\epsilon_1\) we deduce that \(b_{i_j}>0\) for every \(j=1,\ldots,r\).
Now if  \(A\) is a subspace which in the graph is in  level 1, then it contains some of the leaves, say  \(<\alpha_{i_1}>,...,<\alpha_{i_q}>\). By the maximality of \(\MS\) we deduce that \(dim \ A=q+1\) and we can write \(A=<\alpha_h,\alpha_{i_1}, ...,\alpha_{i_q}>\) where  \(\alpha_h\)  is the only simple root which belongs to \(A\) but does not belong to the leaves of the graph.  Then  
\[\pi_A=c_h\alpha_h+ \sum_{j=1}^{q}a_{j}\pi_{<\alpha_{i_j}>}\]
with   \(c_h>0\)  and \(a_{j}\leq R^{q+1}_1\) for every \(j=1,...,q\). Therefore 
\[ (v_\MS,\pi_A)=\epsilon_{q+1}= c_h (v_\MS,  \alpha_h)+ \sum_{j=1}^{q}a_{j} (v_\MS,   \pi_{<\alpha_{i_j}>})\leq c_h(v_\MS,  \alpha_h)+    qR^{q+1}_1\epsilon_1\]
From Lemma \ref{lemmaepsilon} we know that 
\[\epsilon_{q+1}>   (q+1)R^{q+1}_1 \epsilon_1\]
Then \( c_h(v_\MS, \alpha_h) \) must be \(>0\) and from this we deduce, given that   \(c_h>0\) and \((v_\MS,  \alpha_h)\) is a positive multiple of \(b_h\),  that  \(b_h>0\). In a similar  way, by induction on the level, we  prove that all the coefficients \(b_i\) are \(>0\).

\end{proof}
\begin{prop}
\label{teonestoedro}

Let us consider a maximal nested set  \(\MS\) of \(\G_{fund}\).  For every \(A\in \G_{fund}-\MS\) the vector   \(v_\MS\) belongs to the open part of \(HS_A\), therefore   \(v_\MS\) is a vertex of  \(P_{\G_{fund}}\). 

\end{prop}

\begin{proof}
We know by definition that \(v_\MS\) belongs to the hyperplanes \(H_\Gamma\) for every \(\Gamma \in \MS\), therefore to prove  the claim it is sufficient to show that for every \(A\in \G_{fund}-\MS\) the vector  \(v_\MS\) belongs to the open part of \(HS_A\).

The set \(\{A\}\cup \MS\) is not nested since \(\MS\) is a maximal nested set and \(A\) doesn't belong to \(\MS\). 

Let \(C\) be the minimal element in \(\MS\) which strictly contains \(A\) (it could be \(C=V\)).
We observe that there is one (and only one, by the maximality of \(\MS\))  simple root \(\alpha_i\) which belongs to  \(C\) but doesn't belong to any \(K\in \MS\) such that \(K\subsetneq C\). Then \(\alpha_i\) must belong to \(A\): if \(\alpha_i\notin A\), we have   \(A\subseteq T=\sum_{K\in \MS, K\subsetneq C, K\cap A\neq \{0\}}K\). 
We notice that \(T\) strictly includes the subspaces \(K\) such that  \(K\in \MS, K\subsetneq C, K\cap A\neq \{0\}\) because of  the minimality of \(C\). Now, since \(A+T\in \G_{fund}\)\footnote{For any  \(K\) in the sum, \(A+K\in \G_{fund}\) since \(A,K\in \G_{fund}\) and \(A\cap K\neq \{0\}\);  for the same reason,  if we take \( K_1\neq K\) in the sum we have \(A+K+K_1\in \G_{fund}\), and so on..}  this implies  that  \(T\in \G_{fund}\)  which contradicts the nestedness of \(\MS\).

Since \(\alpha_i \in A\) we can find a subset \(\mathcal{I}\) of \(\{K\in \MS \: | \:  K\subsetneq C\}\) such that 
\[C=A+\sum_{K\in \mathcal{I}}K\]
and the sum is non-redundant.
Then we have  
\[ R^{dim \ C}_{dim \ A}\pi_A+\sum_{K\in \mathcal{I}} R^{dim \ C}_{dim \ K} \pi_K\gtrdot \pi_C\]
which implies 
\[\pi_A\gtrdot  \frac{1}{R^{dim \ C}_{dim \ A}}\left( \pi_C-\sum_{K\in \mathcal{I}} R^{dim \ C}_{dim \ K}\pi_K\right )\]

Now  since \(v_\MS\) belongs to the open fundamental chamber (Proposition \ref{basenested}) we have
\[(v_\MS, \pi_A)\geq 
\frac{1}{R^{dim \ C}_{dim \ A}}\left( \epsilon_{dim \ C}-\sum_{K\in \mathcal{I}} R^{dim \ C}_{dim \ K}\epsilon_{dim \ K}\right )\]
We observe that  \(v_\MS\) belongs to the open part of \(HS_A\) if and only  \((v_\MS,\pi_A)>\epsilon_{dim \ A}\); this inequality is verified 
 since,  by Lemma \ref{lemmaepsilon} we have:
\[\epsilon_{dim \ C}> R^{dim \ C}_{dim \ A}\epsilon_{dim \ A}+\sum_{K\in \mathcal{I}} R^{dim \ C}_{dim \ K} \epsilon_{dim \ K}\]

 \end{proof}

The proof of Theorem \ref{teo:nestoedro} is an immediate consequence of the Propositions \ref{prop:nonappartiene}, \ref{basenested}, \ref{teonestoedro}: the faces of dimension \(i\)  are in bijection with the nested sets of cardinality \(n-i\) containing \(V\).

\section{From nestohedra to the permutonestohedron}
\label{sec:permutodimostrazioni}
This section is devoted to the proof of Theorem \ref{teo:verticipermutoedro}. We will split the proof in two steps, given by   following propositions:
\begin{prop}
\label{appartenerealpermutonestoedro}
For every  \(\sigma \in W\),  for every \(A \in \G_{fund}-\{V\}\)  and for every maximal nested set \(\MS\) of \( \G_{fund}\), we have
\[(\delta_A,\sigma v_\MS)\leq (\delta_A,v_\MS)\]
and the equality holds if and only if \(\sigma \in W_A\). 
This means that  \(\sigma v_\MS\) belongs to \(HS_A\), and  it lies in  \(H_A\) if and only if \(A\in \MS \) and \(\sigma \in W_A\); more precisely, \(H_A\) is the affine span of such vectors.
 Furthermore we have 
\[(\delta,\sigma v_\MS)< (\delta,v_\MS)\]
for every \(\sigma\in W\) different from the identity.

\end{prop}

\begin{proof}
Let us consider \(A \in \G_{fund}-\{V\}\) (in the case \(A=V\) the proof is similar).
Let \(I=\{\alpha_{i_1}, \alpha_{i_2},\ldots,  \alpha_{i_k}\}\) be the set of simple roots that  belong to \(A\).
Then  \(\delta_A=\sum_{i\in \Delta-I}b_i\omega_i\) with all the \(b_i>0\) (Remark \ref{pia}).


Since \(v_\MS\) lies inside the open fundamental chamber, we can write 
\(v_\MS=\sum_{j\in \Delta}a_j\omega_j\) with all the \(a_j>0\).
As it is well known (see for instance \cite{bourbaki4}, \cite{humphreyscoxeter}), 
\[\sigma(\omega_i)=\omega_i-\sum_{t=1}^{s_i}\beta_{it}\]
where the \(\beta_{it}\) are positive roots and \(s_i\in \N\): as a notation, when \(\sigma(\omega_i)=\omega_i\) we put \(s_i=0\) and the sum is empty.
Then we have:

\[(\delta_A,\sigma v_\MS)= (\sum_{i\in \Delta-I}b_i\omega_i, \sum_{j\in \Delta}a_j (\omega_j-\sum_{t=1}^{s_j}\beta_{jt}))=\]
\[=(\delta_A,v_\MS)-(\sum_{i\in \Delta-I}b_i\omega_i, \sum_{j,t}\beta_{jt})\]
The  scalar product \[(\sum_{i\in \Delta-I}b_i\omega_i, \sum_{j,t}\beta_{jt})\] is easily seen to be \(\geq 0\) since \((\omega_i,\beta_{jt})\geq 0\). 

If \(\sigma \in W_A\) then all the roots \(\beta_{jt}\) belong to \(A\), therefore 
\[(\sum_{i\in \Delta-I}b_i\omega_i, \sum_{j,t}\beta_{jt})=0.\] 
Otherwise,  at least one of the positive roots, say  \(\beta_{rs}\), does not belong to \(A\),\footnote{In fact let \(\sigma=\sigma_{\alpha_{i_1}}\sigma_{\alpha_{i_2}} \cdots \sigma_{\alpha_{i_k}} \) be a reduced expression for \(\sigma\), and let \(r\) be the smallest  index such that \(\alpha_{i_r}\) does not belong to \(I\).  Then  for every \(j\) the roots \(\beta_{jt}\) which appear in the proof are among  the roots:
\[\sigma_{\alpha_{i_1}}\sigma_{\alpha_{i_2}} \cdots \sigma_{\alpha_{i_{k-1}}}\alpha_{i_k}, \quad   \sigma_{\alpha_{i_1}}\sigma_{\alpha_{i_2}} \cdots \sigma_{\alpha_{i_{k-2}}}\alpha_{i_{k-1}},\quad  \ldots \quad,  \sigma_{\alpha_{i_1}}\sigma_{\alpha_{i_2}} \cdots \sigma_{\alpha_{i_{r-1}}}\alpha_{i_{r}}, ....\]
In particular  the root
\(\sigma_{\alpha_{i_1}}\sigma_{\alpha_{i_2}} \cdots \sigma_{\alpha_{i_{r-1}}}\alpha_{i_{r}}\) is one of the roots \(\beta_{r,t}\) and it satisfies
\[(\sigma_{\alpha_{i_1}}\sigma_{\alpha_{i_2}} \cdots \sigma_{\alpha_{i_{r-1}}}\alpha_{i_{r}},\omega_{i_r})=(\alpha_{i_{r}},\omega_{i_r})>0\]} 
and we have:
\[(\sum_{i\in \Delta-I}b_i\omega_i, \beta_{rs})>0\] 
It remains to prove that \(H_A\) is the affine span of the vectors \(\sigma v_\MS\) with \(\sigma \in W_A\) and \(A\in \MS\).
The  vectors  \(v_\MS\) with   \(A\in  \MS\)  are all the vertices   of a (\(n-2\))-dimensional face of the nestohedron \(P_{\G_{fund}}\). The affine span of this face  is the    subspace  \(T\) whose elements satisfy the relations \((x,\delta_{})=a\) and  \((x,\pi_{A})=\epsilon_{dim \ A}\).  Then   \(T+A\) coincides with  the  hyperplane \(H_A\) defined by the  relation \((x,\delta_A)=a-\epsilon_{dim \ A}\).

Now the   vectors \(\sigma v_\MS\) with \(A\subset \MS\) and \(\sigma\in W_A\) lie on  \(T+A\) and, since  for every simple root \(\alpha_i\) which  belongs to \(A\), we have that \(\sigma_{\alpha_i} v_\MS -v_\MS\)  is a non zero scalar multiple of \(\alpha_i\), the affine span of these vectors  coincides with \(T+A=H_A\).
\end{proof}

\begin{lemma}
\label{piB}
Let \(B\) be a subspace which does not belong to \(\G_{fund}\) and is generated by some of the simple roots and let  \(B=A_1\oplus A_2\oplus \cdots \oplus A_k\) be its  \(\G_{fund}\)-decomposition.
Then the  subspaces \(A_1,  A_2, \cdots , A_k\) are pairwise orthogonal. 
The vector \(\delta_B=\delta-\sum_{i=1}^k\pi_{A_i}\) of Definition \ref{definHB} can be obtained as 
\[\delta_B=\frac{1}{| W_{A_1}\times W_{A_2} \times \cdots  \times W_{A_k}  |}\sum_{\sigma \in W_{A_1}\times W_{A_2} \times \cdots  \times W_{A_k}} \sigma \delta\]
Furthermore, if  \(I\) is  the set of indices given  by the  \(i\in\{1,2,...,n\}\) such that \(\alpha_i\in B\), we have  
\[\delta_B=\sum_{s\in \{1,2,...,n\}-I}c_s\omega_s\]  with all \(c_s>0\).

\end{lemma}

\begin{proof}
As for the orthogonality of the subspaces \(A_i\), we notice that if  two simple roots \(\alpha, \beta\) satisfy  \(\alpha \in A_1\), \(\beta\in A_2\) and \((\alpha,\beta)<0\), then \(<\alpha,\beta>\) is an irreducible subspace. Therefore   it belongs to \(\G_{fund}\), and this contradicts that \(B=A_1\oplus A_2\oplus \cdots \oplus A_k\) is a \(\G_{fund}\)-decomposition.
The claims on \(\delta_B\) easily follow  from the orthogonality of the subspaces \(A_i\)  and from the formula for   the vectors  \(\pi_A\)  in   Remark \ref{pia}.
\end{proof}

\begin{prop}
\label{pianispeciali}
Let \(B\) be a subspace which does not belong to \(\G_{fund}\) and is generated by some of the simple roots and let  \(B=A_1\oplus A_2\oplus \cdots \oplus A_k\) be its  \(\G_{fund}\)-decomposition. Then for every  \(\sigma \in W\) and for every maximal nested set \(\MS\) of \( \G_{fund}\) the vertices    \(\sigma v_\MS \) lie in the half-space \(\overline{ HS}_B\). They lie  on the  hyperplane \(\overline{H}_B\) if and only if   \(\{A_1,A_2,...,A_k\}\subset \MS\) and \(\sigma\in W_{A_1}\times W_{A_2} \times \cdots \times W_{A_k}\) and     \(\overline{H}_B\) is the affine span of all such vertices.
\end{prop}

\begin{proof}
The vertices \(v_\MS\) with   \(\{A_1,A_2,...,A_k\}\subset \MS\)  are all the vertices   of a (\(n-1-k\))-dimensional face of the nestohedron \(P_{\G_{fund}}\). The affine span of this face  is the    subspace  \(T\) whose elements satisfy the relations \((x,\delta_{})=a\) and  \((x,\pi_{A_j})=\epsilon_{dim \ A_j}\) for every \(j=1,...,k\). Then   \(T+B\) is the  hyperplane\footnote{One can check that the  dimension of \((T-v_\MS)+ B\) is equal to \(n-1\) by checking that the dimension of \((T-v_\MS)\cap B\) is  \(dim \ B -k\). In fact  a vector in \(B\)  belongs   to \((T-v_\MS)\) if and only if it satisfies the independent equations \((x,\pi_{A_j})=0\) for every \(j=1,...,k\) (from Lemma \ref{piB} one deduces that  the equation \((x,\delta)=0\) is dependent on  these for the vectors in \(B\)).  } \(\overline{H}_B\) whose  elements are subject to the  relation \((x,\delta_B)=a-\epsilon_{dim \ A_1}-\epsilon_{dim \ A_2}- \cdots - \epsilon_{dim \ A_k}\).

As shown in  Lemma \ref{piB}, \(\delta_B\) is a scalar multiple of \(\sum_{\sigma \in W_{A_1}\times W_{A_2} \times \cdots \times W_{A_k}} \sigma \delta\). Therefore    for every \(\sigma\in W_{A_1}\times W_{A_2} \times \cdots \times W_{A_k}\) we have 
\[(  \sigma v_\MS,\delta_B)= (  v_\MS,\sigma^{-1}\delta_B)=(  v_\MS,\delta_B)\]
and  the   vectors \(\sigma v_\MS\) with \(\sigma\) as above lie on  the hyperplane \(\overline{H}_B\).

  Their affine span is exactly \(T+B=\overline{H}_B\) since the vertices  \(v_\MS\) span \(T\) and, for every simple root \(\alpha_i\) which  belongs to \(B=A_1\oplus A_2\oplus \cdots \oplus A_k\), we have that \(\sigma_{\alpha_i} v_\MS -v_\MS\)  is a non zero scalar multiple of \(\alpha_i\). 

Let us now  prove that for every \(\gamma\in W\) and for every maximal nested set  \(\T\) in \(\G_{fund}\) which does not contain  \(\{A_1,A_2,...,A_k\}\) we have:
\[(\gamma v_\T,\delta_B)<a-\epsilon_{dim \ A_1}-\epsilon_{dim \ A_2}- \cdots - \epsilon_{dim \ A_k}\]
This  inequality is equivalent to 
\[(\gamma v_\T,\pi_B)>\epsilon_{dim \ A_1}+\epsilon_{dim \ A_2}+ \cdots + \epsilon_{dim \ A_k}\]
where \(\pi_B=\delta-\delta_B=\pi_{A_1}+ \cdots + \pi_{A_k}\).
We first prove this when \(\gamma=e\).

For every \(i=1,2,...,k\) let  \(C_i\) be the minimal subspace in \(\T\) which contains \(A_i\). We notice that at least for one index \(j\) we have  \(C_j\neq A_j\) because  \(\T\)   does not contain  \(\{A_1,A_2,...,A_k\}\).
 Then as in the proof of Proposition  \ref{teonestoedro} we deduce that 
\((v_\T,\pi_{A_j})> \epsilon_{dim \ A_j}\)
 because the list  of the numbers  \(\epsilon_j\) is suitable.  This easily leads to prove that  \((v_\T,\pi_B) >\epsilon_{dim \ A_1}+\epsilon_{dim \ A_2}+ \cdots + \epsilon_{dim \  A_k}\).

When \(\gamma\neq e\), since \(\delta_B=\sum_{i \ s.t. \alpha_i\notin B}c_i\omega_i\)  with all \(c_i>0\) (Lemma \ref{piB}) we can conclude, as in the first part of the proof of Proposition  \ref{appartenerealpermutonestoedro},  that 
\[(\gamma v_\T ,\delta_B)\leq (v_\T,\delta_B).\]

It remains to prove the claim when the  maximal nested set  \(\T\) contains   \(\{A_1,A_2,...,A_k\}\)   but \(\gamma\notin   W_{A_1}\times W_{A_2} \times \cdots \times W_{A_k}\); this is essentially the same reasoning as in  the second part of the proof of Proposition  \ref{appartenerealpermutonestoedro}. 
\end{proof}

Propositions \ref{appartenerealpermutonestoedro} and \ref{pianispeciali} determine the vertices of the permutonestohedron and  prove  Theorem \ref{teo:verticipermutoedro}.

\section{The face poset of the permutonestohedron}
\label{sec:poset}
In this section we will give a description of the  full  face poset of the permutonestohedron. We will  improve and complete  the  description that was first sketched   in   \cite{gaimrn},  where a non linear realization of the permutonestohedron appeared.




The faces of \(P_\G(\Phi)\) are in bijective correspondence with the pairs \((\sigma H, \MS)\), where: 
\begin{itemize}
\item \(\MS\) is a nested set of \(\G_{fund}\) that  contains \(V\) and has  labels attached to  its minimal elements: if \(A\) is a minimal element, its label is either the subgroup \(W_A\) of \(W\) or the trivial subgroup \(\{e\}\). For brevity in the sequel we will omit to write the label \(\{e\}\)  and we will write \(\underline{A}\) to indicate that \(A\) is labelled by \(W_A\);  

\item \(\sigma H\) is a coset of \(W\), where \(H\) is the subgroup of \(W\) given by the direct product of all the labels.

\end{itemize}

This bijective correspondence is motivated in the following way:  to obtain the face represented by   \((\sigma H, \MS)\) one starts from the face  \(F\) of the nestohedron  \(P(\G_{fund})\) which is associated with the nested set \(\MS\). Then one considers all the images of this face under the action of the elements of \(H\) and takes their convex hull. This gives a face \(F'\) of the permutonestohedron 
which intersects the fundamental chamber (see Proposition \ref{faccedavvero} below). Then \(\sigma F'\) is the face associated with the pair \((\sigma H, \MS)\).

 \begin{prop}
 \label{faccedavvero}
  Let \(\MS\) be  a nested set of \(\G_{fund}\) which contains \(V\) and has labels attached to its minimal elements, with at least one nontrivial label. Let \(A_1\),\(A_2\),...,\(A_k\) the subset of its  minimal elements that  have a nontrivial label and let \(H=W_{A_1}\times W_{A_2} \times \cdots \times W_{A_k}\). Let us then consider the face \(F\) of the nestohedron  \(P(\G_{fund})\) associated with \(\MS\), and let \(F'\) be the convex hull of all the faces \(hF\) with \(h\in H\). Then \(F'\) is the (\(|\MS|-k\))-codimensional face of \(P_{\G}(\Phi)\)  determined  by  the  intersection of the defining hyperplanes associated with \(A_1\oplus A_2\oplus \cdots \oplus A_k\) and \(B+(A_1\oplus A_2\oplus \cdots \oplus A_k)\) for every \(B\in \MS-\{V,A_1, A_2,...,A_k\}\). 
 \end{prop}
 
 \begin{proof}

 First we observe that for every \(B\in \MS-\{V,A_1, A_2,...,A_k\}\) the sum  \(B+(A_1\oplus A_2\oplus \cdots \oplus A_k)\) is equal, by nestedness,  to \(B\oplus A_{i_1}\oplus\cdots \oplus A_{i_r}\) where \(\{A_{i_1},\ldots,A_{i_r}\}\) is a (possibly empty) subset of \(\{A_1, A_2,...,A_k\}\);  if  
 \(\{A_{i_1},\ldots,A_{i_r}\}\) is not empty then the corresponding space is  \({\overline H}_{B\oplus A_{i_1}\oplus\cdots \oplus A_{i_r}}\),  otherwise it is \(H_B\). In both cases it is one of the defining hyperplanes of the permutonestohedron.
 

Let us  then denote by    \(L\)  the face of the permutonestohedron determined by the intersection of the defining hyperplanes mentioned in the claim.

It is straightforward to check, using Propositions \ref{appartenerealpermutonestoedro} and \ref{pianispeciali},  that the   set of vertices of \(L\) coincides with the set of all the vertices which belong to  \(\cup_{h\in H}hF\).

Therefore \(L=F'\); now, applying an argument analogue to the one of the proof of Proposition \ref{pianispeciali} one checks that  the affine span of the vertices in \(F'\) coincides with the  subspace \(<F>+(A_1\oplus A_2\oplus \cdots \oplus A_k)\) that has codimension  \(|\MS|-k\) (here \(<F>\) denotes the affine span of \(F\)).
 \end{proof}

 The pictures in Figure \ref{facce10} and Figure \ref{facce12} illustrate, in the case of the permutoassociahedron \(P_{\F_{A_3}}(A_3)\), the    correspondence between faces and pairs described above.
  \begin{figure}[h]
 \center
\includegraphics[scale=0.22]{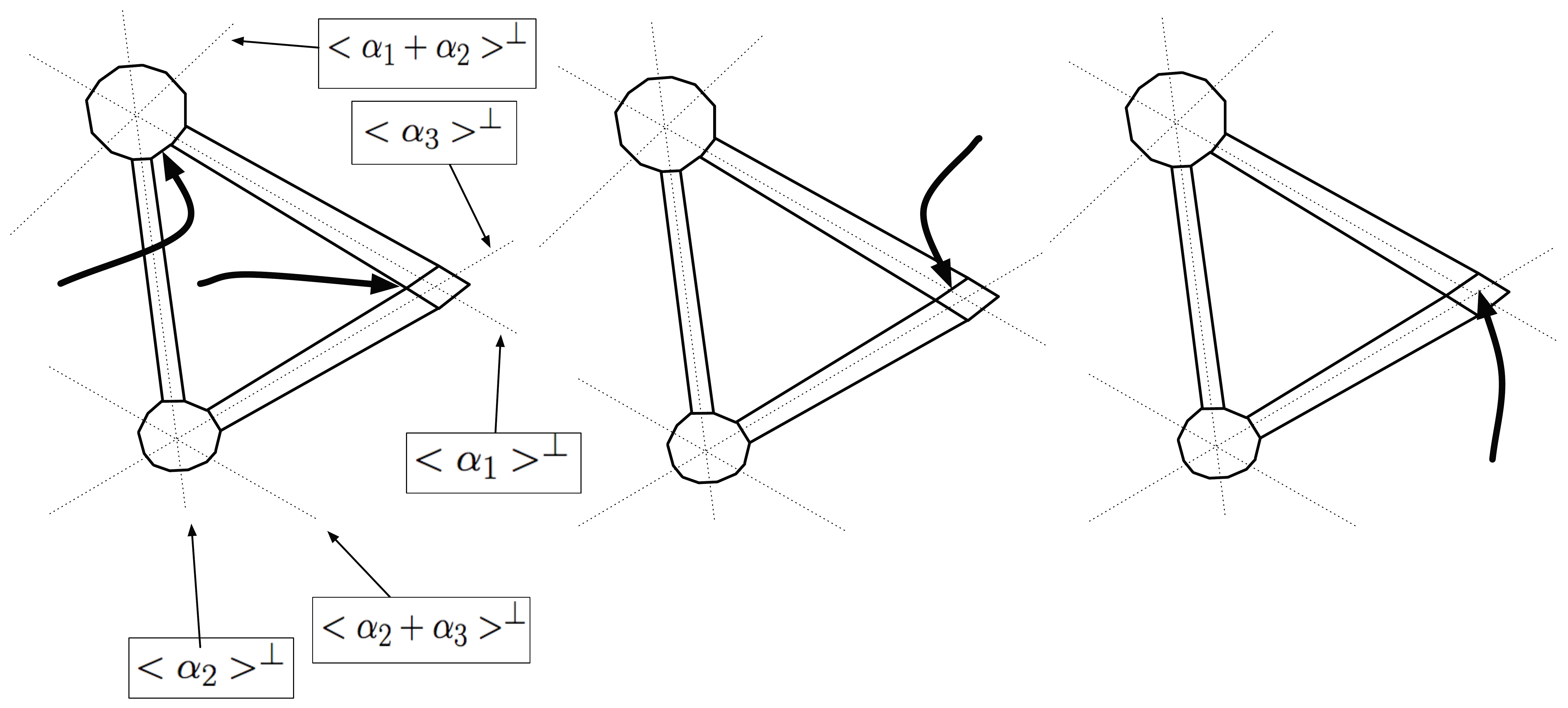}
\caption{Some  planar pictures of the  portion of \(P_{\F_{A_3}}(A_3)\) which is around the fundamental chamber: the dotted lines represent the hyperplanes which intersect the closed fundamental chamber, as indicated in the picture on the left. 
In the picture on the left, the thick  arrows indicate respectively the vertex  \(( \{e\} ,\{V, <\alpha_1>, <\alpha_3>\} )\) and the edge \(( \{e\} ,\{V, <\alpha_1,\alpha_2>\} )\).
In the picture in the middle the thick  arrow indicates the edge \(( W_{<\alpha_1>} ,\{V, \underline{<\alpha_1>}, <\alpha_3>\} )\).
In the picture on the right the thick  arrow indicates the facet \(( W_{<\alpha_1>}\times W_{<\alpha_3>} ,\{V, \underline{<\alpha_1>}, \underline{<\alpha_3>}\} )\).}
\label{facce10}
\end{figure}

 \begin{figure}[h]
 \center
\includegraphics[scale=0.22]{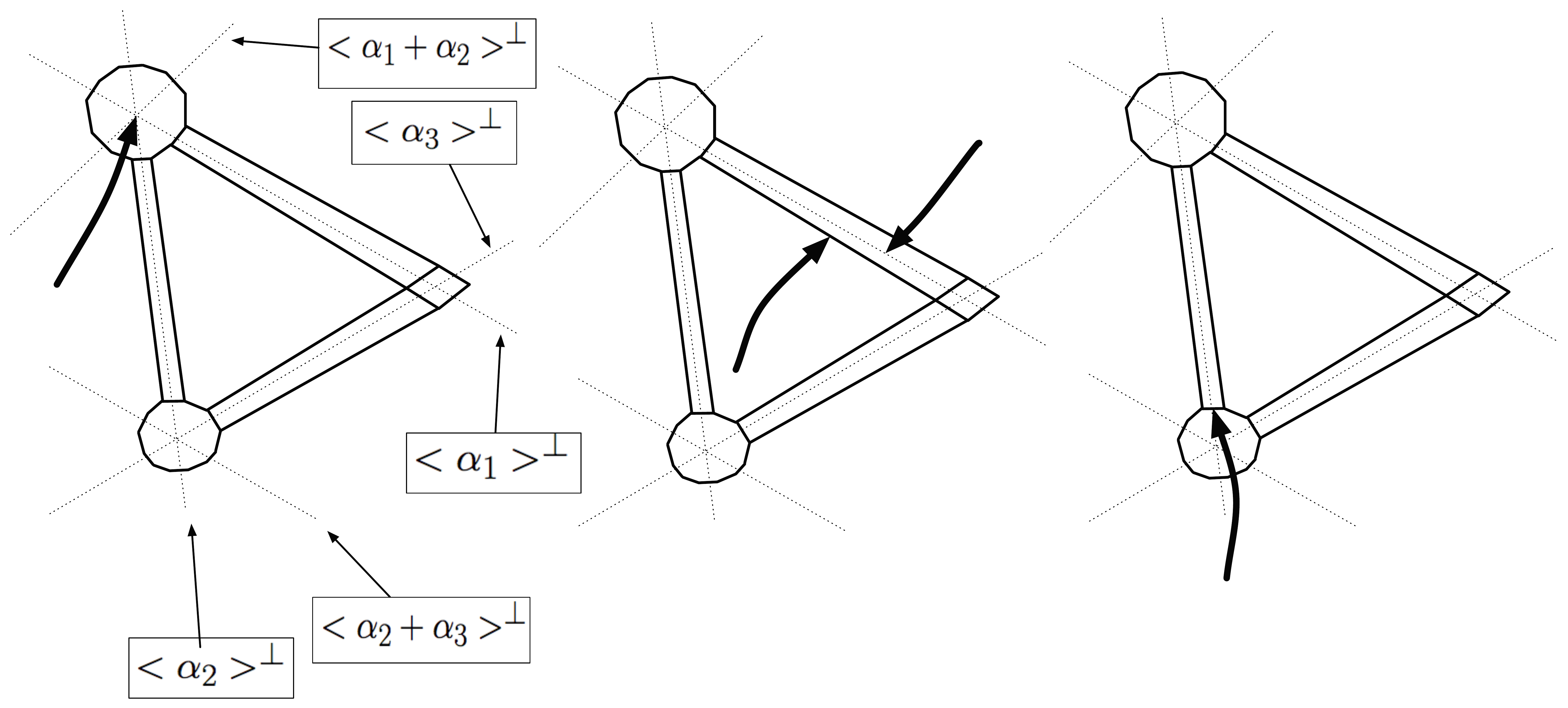}
\caption{Again some  planar pictures of the portion of \(P_{\F_{A_3}}(A_3)\) which is around the fundamental chamber.
In the picture on the left, the thick  arrow indicates the facet  \(( W_{<\alpha_1,\alpha_2>},\{V, \underline{<\alpha_1,\alpha_2>}\} )\).
In the picture in the middle the thick  arrows indicate respectively  the edge \(( \{e\} ,\{V, <\alpha_1>\} )\) and the facet \(( W_{<\alpha_1>} ,\{V, \underline{<\alpha_1>}\} )\).
In the picture on the right the thick  arrow indicates the edge \(( W_{<\alpha_2>},\{V,\underline{ <\alpha_2>}, <\alpha_2, \alpha_3>\} )\).}
\label{facce12}
\end{figure}

From now on we will denote a face by its corresponding pair; as  an immediate  consequence of  Proposition \ref{faccedavvero} we have: 
\begin{cor}
The dimension of the face \((\sigma H, \MS)\) is given by \(n-|\MS|+l\) where \(l\) is the number of nontrivial labels.
\end{cor}
\begin{rem}For instance, \((W, \{\underline{V}\})\) is the full permutonestohedron. 
The nestohedra inside the chambers correspond to the pairs   
\[(\sigma \{e\}, \{V\})\]
for every \(\sigma\in W\),  while the  other facets, the facets that  `cross some of the chambers',   correspond to 
\[(\sigma W_{A_1}\times W_{A_2} \times \cdots \times W_{A_k}, \{V, \underline{A_1}, \underline{A_2},\ldots , \underline{A_k}\})\]

Here the nested set on the right is made by \(V\) and by the proper subspaces \(A_1, A_2, \ldots, A_k\) that are all minimal and with non trivial label.
\end{rem}
The following  corollary points out that, once \(\Phi\) is fixed, only the maximal permutonestohedron \(P_{{\mathcal C}_\Phi}(\Phi)\) is a simple polytope.
\begin{cor}
\label{cor:nonsemplice}
A vertex \(v_\MS\) in \(P_{\G_{fund}}\) belongs to exactly \(n\) facets of \(P_\G(\Phi)\) if and only if \(\MS\) has only one minimal element.
Therefore the  polytope  \(P_\G(\Phi)\) is simple if and only if \(\G\) is the maximal building set \({\mathcal C}_\Phi\).
\end{cor}
\begin{proof}
First we observe that a   nested set \(\MS\) in \(\G_{fund}\)  has only one minimal element if and only if it is linearly ordered.

 Then we notice  that Proposition \ref{faccedavvero} can be used to determine all the facets  that contain \(v_\MS\). These are  the facets  
\[(\sigma W_{A_1}\times W_{A_2} \times \cdots \times W_{A_k}, \{V, \underline{A_1}, \underline{A_2},\ldots , \underline{A_k}\})\]
where \(\{V, A_1, A_2,\ldots , A_k\}\) is a nested subset of \(\MS\) and the \(A_i\) are minimal  (we are including the case \(k=0\), i.e., the face \((\{e\}, \{V\})\)).
If \(\MS\) is not linearly ordered these facets  are more than \(n\), while if \(\MS\) is linearly ordered they are exactly \(n\).

To finish the proof we remark that if  \(\G\) is not the maximal building set \({\mathcal C}_\Phi\) there exist non linearly ordered \(\G\)-nested sets: to obtain a non linearly ordered nested set made by two elements    it is sufficient to take two subspaces \(A,B \in \G\) whose sum is direct and doesn't belong to \(\G\). 
As an immediate  consequence,  if  \(\G\) is not the maximal building set associated to \(\Phi\) there exists a maximal nested set \(\MS\) in \(\G_{fund}\) that is not linearly ordered.

\end{proof}

We now focus on   the facets that   cross the chambers: they are   combinatorially equivalent to a product of a nestohedron with some permutonestohedra. More precisely, 
let us consider the facet \[(\sigma W_{A_1}\times W_{A_2} \times \cdots \times W_{A_k}, \{V, \underline{A_1}, \underline{A_2},\ldots , \underline{A_k}\})\] and  denote by \(\G^{A_i}_{}\) the  subset of \(\G_{}\) given by the subspaces which are included in \(A_i\).  This is a building set associated with  the root system \(\Phi\cap A_i\). Furthermore, let us denote \(A_1\oplus A_2\oplus \cdots \oplus A_k\) by \(D\) and consider the building set \({\overline \G}\) in \(V/D\) given by 
\[{\overline \G}=\{(C+D)/D \: | \: C\in \G_{fund}\}\]
According to the notation  in Section \ref{sec:ricordare} we call \(P_{\overline \G}\) the nestohedron associated with \({\overline \G}\).
\begin{teo}
\label{combinatorialiso}
The facet \((\sigma W_{A_1}\times W_{A_2} \times \cdots \times W_{A_k}, \{V, \underline{A_1}, \underline{A_2},\ldots , \underline{A_k}\})\) of \(P_\G(\Phi)\) is combinatorially equivalent to the product\footnote{Here we are considering the well defined   product of combinatorial polytopes.} 
\[  P_{\overline \G}\times P_{\G^{A_1}_{}}(\Phi\cap A_1) \times \cdots \times P_{\G^{A_k}_{}}(\Phi\cap A_k)\]
\end{teo}

\begin{proof}

Let us show how to associate to a face   \((\sigma H, \MS)\) of  \((\sigma W_{A_1}\times W_{A_2} \times \cdots \times W_{A_k}, \{V, \underline{A_1}, \underline{A_2},\ldots , \underline{A_k}\})\) a face in the product  
\[  P_{\overline \G}\times P_{\G^{A_1}_{}}(\Phi\cap A_1) \times \cdots \times P_{\G^{A_k}_{}}(\Phi\cap A_k)\]
We remark that 
\begin{itemize}
\item[a)]  \(\MS \subset  \G_{fund}\) is a  labelled nested set that contains \(V,A_1, A_2, \ldots, A_k\); 
\item[b)] \(\sigma=\sigma_1\sigma_2\cdots \sigma_k\) with \(\sigma_i\in W_{A_i}\); 
\item[c)]  \(H\)  is a subgroup of \(W_{A_1}\times W_{A_2} \times \cdots \times W_{A_k}\) that  can be expressed as a product of \(W_{A_{ij}}\) for some minimal subspaces \(A_{ij}\in \MS\) (\(i=1,...,k\) and, for every \(i\), \(A_{ij}\subset A_i\) and  the index   \( j\) ranges from  0 to a natural number \(s_i\), with the convention that \(W_{A_{i0}}\) is the trivial group).

\end{itemize}
Then we  associate to \((\sigma H, \MS)\):
\begin{itemize}
\item the face of \(P_{\overline \G}\) which corresponds to the nested set 
\({\overline \MS}=\{K+D/D\: | \: K\in \MS\}\) of    \({\overline \G}\); 
\item for every \(i=1,...,k\), the face 
\[ (\sigma_i W_{A_{i1}}\times W_{A_{i2}} \times \cdots \times W_{A_{is_i}}, \MS^{A_i} )\]
of  \(P_{\G^{A_i}_{}}(\Phi\cap A_i)\), where \(\MS^{A_i}\) is the  subset of \(\MS\) given by the subspaces which are included into \(A_i\) and the labelled subspaces of \(\MS\) keep their label in \(\MS^{A_i}\).
\end{itemize} 
The above described map is easily shown to be bijective, and, using Proposition \ref{faccedavvero}, a poset isomorphism.

\end{proof} 
 As a corollary  of Proposition \ref{faccedavvero} and Theorem  \ref{combinatorialiso}, we conclude this section with an explicit description of the  order relation on the face poset  of \(P_\G(\Phi)\).
\begin{cor}   Given two faces \((\sigma' H', \MS')\) and  \((\sigma H, \MS)\) in the face poset of \(P_\G(\Phi)\)  we   have 
\[(\sigma' H', \MS')< (\sigma H, \MS)\]
if and only if  \(\sigma' H'\subseteq \sigma H\) and \(\MS'\) is obtained from \(\MS\) by a composition of some of the following moves:
\begin{itemize}
\item adding a  subspace which is not minimal, i.e. it contains some of the subspaces in \(\MS\);
\item adding a   subspace \(A\), minimal in \(\MS'\), with trivial label and  with the property that   \(A\)   is not included in any of the minimal subspaces of \(\MS\), or it is included in a minimal subspace \(B\) of \(\MS\) which is  labelled by \(\{e\}\); in the latter case   \(B\)  loses its label; 
\item adding some subspaces \(A_1,..,A_k\)  that are  minimal in \(\MS'\), all  with nontrivial label, and all included in  a minimal subspace \(B\) of \(\MS\) which was   labelled by \(W_B\) and   loses its label;
\item changing the non trivial label of a minimal subspace into the trivial label.
\end{itemize}
\end{cor}




\section{The  automorphism group}
\label{sec:autgroup}
In this section we study the automorphism group  \(Aut(P_\G(\Phi))\), i.e.  the  group of the isometries of \(V\) that  send \(P_\G(\Phi)\) onto itself.
We adopt here the following normalization of the root system \(\Phi\): if \(\Phi\) is made by two or more  irreducible components, we impose that all the short roots have the same length\footnote{If all the roots of an irreducible component have the same length we consider them short roots.}.

Let \(Aut (\Phi)\) be the group of the automorphisms of \(V\) that  leave \(\Phi\) invariant. With the above normalization its elements are isometries.

\begin{teo}
\label{automorfismi}
Let us suppose that \(\G\) is \(Aut (\Phi)\) invariant.
Then \(Aut (\Phi)\) is a subgroup of  \(Aut(P_\G(\Phi))\).  If  \(Aut(P_\G(\Phi))\) leaves invariant the set of the nestohedra \(\{wP_{\G_{fund}}\: | \: w\in W\}\) then \(Aut (\Phi)=Aut(P_\G(\Phi))\).
\end{teo}
\begin{rem}
Until now in this paper we have used the notations \(wv_\MS, wH_A...\), without parentheses,  to indicate the action of an element \(w\) of the Weyl group. We feel that in the following proof this could  be confusing, since the product of elements in the automorphism group comes into play, so we decided, for the sake of clarity, to put parentheses here.
\end{rem}
\begin{proof}
To prove  the first part of the claim we  will show  that an element \(\varphi \in Aut (\Phi)\) permutes the defining hyperplanes of the permutonestohedron. 

Since \(Aut (\Phi)\) is the semidirect product of the Weyl group with the  automorphism group \(\Gamma\) of the Dynkin diagram of \(\Phi\), we can write  \(\varphi=w\gamma\) where \(w\in W\) and \(\gamma\in \Gamma\). 

It is therefore sufficient  to show that \(\varphi\) sends the hyperplanes \(H_V,H_A,{\overline H}_B\), for any \(A \in \G_{fund}\), \(B\in {{\mathcal C}_\Phi}_{fund} -\G_{fund}\),  to other defining hyperplanes of \(P_\G(\Phi)\).

Now \(\gamma(\delta)=\delta\) since \(\gamma\) leaves invariant  the set of the positive roots. Then \(\varphi(\delta)=w(\delta)\), and therefore \(\varphi(H_V)=w(H_V)\)  since the vectors in \(\varphi(H_V)\) satisfy the defining equation  of \(w(H_V)\) which is \((x,w\delta)=a\).

Then we  show that for any \(A\in \G_{fund}\) we have \(\varphi(H_A)=w(H_{\gamma(A)})\): since  \(\delta_A=\delta-\pi_A\) we have  \(\gamma(\delta_A)=\delta-\gamma(\pi_A)\), but \(\gamma(\pi_A)\) is the semisum of all the positive roots contained in \(\gamma(A)\), thus it is equal to \(\pi_{\gamma(A)}\), therefore \(\gamma(\delta_A)=\delta_{\gamma(A)}\). It is now immediate to prove   that \(\varphi(H_A)\) satisfies the defining equation of \(w(H_{\gamma(A)})\).

The same reasoning applies to prove that, for \(B\in {{\mathcal C}_\Phi}_{fund} -\G_{fund}\) we have  \(\varphi({\overline H}_B)=w({\overline H}_{\gamma (B)})\).

For the second part of the claim, let us suppose that  \(\theta \in Aut(P_\G(\Phi))\) sends  \(P_{\G_{fund}}\) to \(w(P_{\G_{fund}})\).
 It is sufficient  to show that \(w^{-1}\theta\) belongs to \(Aut (\Phi)\). 
 First we observe that since \(w^{-1}\theta\) sends \(H_V\) onto itself and is an isometry we have \(w^{-1}\theta(\delta)=\delta\).
 Now, let us consider a simple root \(\alpha \in \Delta\) and let \(\MS\) be a nested set which contains \(<\alpha>\), so that \(v_\MS\in H_{<\alpha>}\). 
 
 We observe that \(w^{-1}\theta (v_\MS)\) belongs to \(P_{\G_{fund}}\); it also  belongs  to 
\(H_{<\beta>}\)  where \(\beta \) is a simple root. 
In fact if we denote by \(F\) the facet of \(P_{\G_{fund}}\) determined by a hyperplane  \(H_{A}\) with \(A\in \G_{fund}\) (the corresponding pair in the poset is  \((\{e\}, \{V, A\})\)), and by \({\overline F}\) the facet of the permutonestohedron different from \(P_{\G_{fund}}\) which contains \(F\)  (the corresponding pair  in the poset is  \((W_{A}, \{V, \underline{A}\})\)), we notice that in \({\overline F}\)  there are exactly \(|W_A|\) (\(n-2\))-dimensional faces which belong to one of the nestohedra \(\{w(P_{\G_{fund}})\: | \: w\in W\}\).  As a consequence, under our hypothesis, also the   the facet \(w^{-1}\theta ({\overline F})\) has \(|W_A|\) (\(n-2\))-dimensional faces which belong to one of the nestohedra \(\{w(P_{\G_{fund}})\: | \: w\in W\}\).

Moreover we observe that   \(A\) is the span of a simple root  if and only if  \(|W_A|=2\). 

Therefore in our case we have  \(A=<\alpha>\) and  \(w^{-1}\theta ({\overline F})\) has two   (\(n-2\))-dimensional faces which belong to one of the nestohedra \(\{w(P_{\G_{fund}})\: | \: w\in W\}\).  This means that  \(w^{-1}\theta (F)\) is    a facet  of \(P_{\G_{fund}}\) determined by a hyperplane \(H_{B}\)   with \(B\) equal to the span \(<\beta>\) of  a simple root \(\beta\) (possibly equal to \(\alpha\)).

Since \(w^{-1}\theta\) sends \(H_{<\alpha>}\) onto \(H_{<\beta>}\) and is an isometry we have that 
\(w^{-1}\theta(\delta_{<\alpha>})=\delta_{<\beta>}\). From this, since \(w^{-1}\theta(\delta)=\delta\) and \(\delta_{<\alpha>}=\delta-\frac{1}{2}\alpha\), \(\delta_{<\beta>}=\delta-\frac{1}{2}\beta\), it follows \(w^{-1}\theta(\alpha)=\beta\).

This means that \(w^{-1}\theta\) sends \(\Delta\) to itself. 

Now,  considering a two dimensional subspace  \(D\in \G_{fund}\)  and comparing the cardinality of \(W_D\) and \(W_{w^{-1}\theta (D)}\)   we prove that  two simple roots \(\alpha\) and \(\beta\) are orthogonal if and only if their images  \(w^{-1}\theta (\alpha) \) and \(w^{-1}\theta (\beta ) \) are orthogonal.
Moreover  \(w^{-1}\theta\) preserves  root lengths, since sends \(\Delta\) to itself and is an isometry.  
These properties imply that \(w^{-1}\theta\) is an automorphism of the Dynkin diagram.
\end{proof}
\begin{rem}
If \(\Phi=A_n\), any \(W\)-invariant building set is also \(Aut(\Phi)\) invariant and satisfies the hypothesis of the  theorem above. 
If a building set is not \(Aut(\Phi)\) invariant, let \(G\) be the maximal subgroup of \(Aut(\Phi)\) which leaves \(\G\) invariant. Then the claim of the theorem (and its proof) is still valid  with \(G\) in place of \(Aut(\Phi)\). 
\end{rem}
As we will see in the next section, we can choose a suitable list \(\epsilon_1<\epsilon_2\) such that  \(Aut( P_{\F_{A_2}}(A_2; \epsilon_1,\epsilon_2))\)  is greater than \(Aut \ A_2\). 
Anyway, we can state the following theorem.
\begin{teo}
Let \(\G\) be \(Aut (\Phi)\) invariant.
There are infinite suitable lists \(\epsilon_1< \cdots < \epsilon_n=a\) such that 
\(Aut (\Phi)=Aut(P_\G(\Phi))\).
More precisely, once \(a\) is fixed, for all the possible suitable lists whose greatest number is \(a\), except  at most  for a finite number, we have \(Aut (\Phi)=Aut(P_\G(\Phi))\).
\end{teo}
\begin{proof}
This follows from the observation that   the elements of \(Aut(P_\G(\Phi))\) are isometries   and, once \(a\) is fixed, all the  choices, except for a finite number of exceptions,  of the other numbers \(\epsilon_i\)  imply  that   the distance from the origin of \(H_V\) is different from the distances from the origin of the other defining  hyperplanes. 
Therefore   \(Aut(P_\G(\Phi))\) leaves invariant the set of the nestohedra \(\{wP_{\G_{fund}}\: | \: w\in W\}\) and we can apply Theorem \ref{automorfismi}.
\end{proof}

The following corollary  illustrates another sufficient (non metric)  condition for \(Aut(P_\G(\Phi))= Aut (\Phi)\). For any \(C \in  {{\mathcal C}_\Phi}_{fund}\) let  \(G_C\) be the subgroup of \(Aut(P_\G(\Phi))\) which leaves the facet determined by the hyperplane \(H_C\) (or \({\overline H}_C\))  invariant. We observe that \(W_C\subset G_C\) by construction. 
\begin{cor}
\label{criterioaut}
Let \(\G\) be \(Aut (\Phi)\) invariant.
If for every \(C \in  {{\mathcal C}_\Phi}_{fund}\) the automorphism group of the Dynkin diagram of \(\Phi\) is not isomorphic to   the group \(G_C\)  then we have \(Aut(P_\G(\Phi))= Aut (\Phi)\). 
In particular if the automorphism group of the Dynkin diagram of \(\Phi\) is trivial then \(Aut(P_\G(\Phi))= Aut (\Phi)=W\).
\end{cor}
\begin{proof}
This is an immediate application of Theorem \ref{automorfismi} since   the subgroup of \(Aut(P_\G(\Phi))\) which leaves \(P_{\G_{fund}}\) invariant is, as we showed in the proof of Theorem \ref{automorfismi}, the automorphism group \(\Gamma\) of the Dynkin diagram. Let \(\varphi\in Aut(P_\G(\Phi))\); then  the subgroup of \(Aut(P_\G(\Phi))\) which  leaves the facet \(\varphi(P_{\G_{fund}}) \) invariant is a conjugate of \(\Gamma\), and, under our hypothesis, this implies that  \(\varphi(P_{\G_{fund}}) \) is one of the facets in \(\{wP_{\G_{fund}}\: | \: w\in W\}\). The claim of Theorem \ref{automorfismi} then implies that  \(Aut(P_\G(\Phi))= Aut (\Phi)\).
\end{proof}

\begin{rem}
\label{rem:facecounting}
The corollary above  is easy to apply when \(\Phi\) is irreducible,  since the automorphism group of the Dynkin diagram is  small (it is either trivial or \(\Z_2\),  or \(S_3\) in the case \(D_4\)) and it is easy to compare it with the groups \(G_C\). 
\end{rem}

\section{A remark on the \(S_{n+1}\) action on the face poset of \(CY_{\F_{A_{n-1}}}\)}
\label{sec:extendedaction}
As we have seen in the preceding section, the face poset  of a permutonestohedron \(P_\G(\Phi)\)  provides  nice geometrical realizations of  representations of \(W\) or even  of \(Aut (\Phi)\).  In this section we focus on a special case, where  \(W=S_n\) and  some  representations  of \(S_{n+1}\) also come into play.

First we recall that there  is a well know `extended' \(S_{n+1}\) action on the De Concini-Procesi model \(Y_{\F_{A_{n-1}}}\), that is a quotient of \(CY_{\F_{A_{n-1}}}\): it comes from the isomorphism with the moduli space \(M_{0,n+1}\) (see \cite{DCP1}, \cite{gaiffiseminario}), and its character has been computed in \cite{etihenkamrai}.

This extended action can be lifted  to the face poset of \(CY_{\F_{A_{n-1}}}\), which is a subposet of \(P_{\F_{A_{n-1}}}(A_{n-1})\). 
We illustrate this lifting using our description of this subposet.

Let \(\Delta=\{\alpha_0,\alpha_1,...,\alpha_{n-1}\}\) be a basis for the root system of type \(A_n\), where we added to a basis of \(A_{n-1}\) the extra root \(\alpha_0\),  and let \(\tilde{\Delta}=\{\tilde{\alpha},\alpha_0,\alpha_1,...,\alpha_{n-1}\}\) be the set of   roots that appear in  the  affine diagram. 
We identify in the standard way  \(S_{n+1}\) with the group which permutes \(\{0,1,...,n\}\) and    \(s_{\alpha_0}\) with the transposition \((0,1)\). Therefore    \(S_{n}\), the subgroup generated by \(\{s_{\alpha_1},...,s_{\alpha_{n-1}}\}\), is identified with the subgroup which permutes \(\{1,...,n\}\).

Let \(\Sc=\{V, A_1,A_2,...,A_k, B_1,B_2,...,B_s\}\) be a nested set  in \({\F_{A_{n-1}}}_{fund}\) and let \(\sigma \in S_{n+1}\).  Let us then denote by \(C\) the cyclic subgroup generated by \((0,1,2,3,4,5,...,n)\) 
and by  \(w=\sigma(0,1,2,3,4,5,...,n)^r\)  the only element  in  the coset \(\sigma C\) which fixes \(0\); we notice that  \(w\) belongs to   \(S_n\).

Moreover, let us suppose that, for every subspace \( A_j\), some of the roots contained in \(\sigma A_j\) have  \(\alpha_0\) in their support (when they are   written with respect to the basis \(\Delta\)), while this doesn't happen for the subspaces 
\(\sigma B_t\).
Then for every \(j\) we   denote  by \(\overline A_j\) the subspace generated by all the roots of \(\tilde{\Delta}\)  which are orthogonal to \(A_j\).

\noindent As a first step in the description of  the \(S_{n+1}\) action, we put   \[\sigma \cdot ( \{e\},\Sc  )= (w\{e\}, \{V, ..,w^{-1}\sigma\overline{A_j},.., w^{-1}\sigma B_s,..\}).\]
  As one can quickly check, this can be extended  to an  \(S_{n+1}\) action on  the full face poset of \(CY_{\F_{A_{n-1}}}\) by imposing that \(\sigma\) sends  
 the  face \((\gamma\{e\}, \Sc)\), where \(\gamma \in S_{n}\), to the face \(\sigma \gamma \cdot  ( \{e\},\Sc  )\).

\begin{es}  Let \(\Sc\) be the nested set of  \({\F_{A_{4}}}_{fund}\) made by \(V\), \(A=<\alpha_1,\alpha_2>\) and \(B=<\alpha_4>\). The group  \(S_6\) is generated by the reflections \(s_{\alpha_0}, s_{\alpha_1},s_{\alpha_2},s_{\alpha_3},s_{\alpha_4}\) and we identify \(S_5\) with the subgroup generated by \(s_{\alpha_1},s_{\alpha_2},s_{\alpha_3},s_{\alpha_4}\).
Now  we  compute \(s_{\alpha_0}( \{e\},\{V, A,B\}  )\).

We notice that the root \(s_{\alpha_0}\alpha_1\) contains \(\alpha_0\) in its  support (when it  is written with respect to the basis \(\Delta\)).
We then denote by \(\overline A\) the subspace generated by all the roots of \(\tilde{\Delta}\)  which are orthogonal to \(A\): \(\overline A=<\tilde{\alpha},\alpha_4>\).

Let \(w=(0,1)(0,1,2,3,4,5)\) i.e. the representative of the coset \((0,1)C\) in \(S_6\) which leaves \(0\) fixed.
Then \(s_{\alpha_0}=(0,1)\) sends  the face \(( \{e\},\{V,A,B\}  )\) to the face \[(w\{e\}, \{V, w^{-1}s_{\alpha_0}\overline{A}, w^{-1}s_{\alpha_0}B\})=(w\{e\}, \{V,<\alpha_3,\alpha_4>, <\alpha_3>\}) \]
\end{es}

Therefore the group \(S_{n+1}\) acts on the face poset  of \(CY_{\F_{A_{n-1}}}\), which  is the disjoint union of \(n!\) associahedra: the action of \(\sigma\in S_{n+1}\) sends the face poset of  the  associahedron that  lies in the fundamental chamber   onto the face poset of an  associahedron that   lies in a  chamber which may be different from the fundamental one.   
But it is easy to provide  examples where  two associahedra that  lie in two adjacent chambers are sent to two associahedra whose chambers are not  adjacent. 

This shows that  this lifted action  is not induced by an isometry and  it cannot be extended to the full permutoassociahedron.
Anyway  it provides  geometrical realizations of all the representations \(Ind_G^{S_n+1} Id\),  where  \(G\) is any  subgroup of the cyclic group \(C\) and \(Id\) is  its trivial representation  (the case \(G=\{e\}\), i.e. the regular representation of \(S_{n+1}\),  occurs only if \(n\geq 4\)). 

In fact the stabilizer of an element of the face poset is by construction a subgroup of the cyclic group \(C\): in the notation above, \(w=e\) only if \(\sigma\in C\).
Now we want to  show that all the above mentioned representations appear.
As for the regular representation of \(S_{n+1}\), we notice that,  for instance, if \(n\geq 4\)
 the stabilizer of \(( \{e\},\{V,<\alpha_1,...,\alpha_{n-2}>\}  )\) is the trivial subgroup. 

Let then  \(d<n+1\) be a divisor of \(n+1\). We will exhibit an element of the face poset whose   stabilizer is generated by \((0,1,2,3,4,5,...,n)^d\).

If \(d>2\) and \(dk=n+1\) we consider for instance  the face  \(( \{e\},\{V,<\alpha_1,...,\alpha_{d-2}>, <\alpha_{d+1},...,\alpha_{2d-2}>,..., <\alpha_{(k-1)d+1},...,\alpha_{kd-2}>\}  )\): its stabilizer is generated by \((0,1,2,3,4,5,...,n)^d\). 

If \(d=2\) and \(2k=n+1\) we consider \(( \{e\},\{V,<\alpha_1,\alpha_2,..., \alpha_{n-2}>, <\alpha_1>,<\alpha_3>,...,<\alpha_{n-2}>\}  )\). Its stabilizer is generated by \((0,1,2,3,4,5,...,n)^2\). 

If \(d=1\) we consider \(( \{e\},\{V\}  )\). Its stabilizer is, by definition of the \(S_{n+1}\) action, the full cyclic group \(C\).

\section{Examples}
\label{sec:examples}
\subsection{Some low-dimensional cases}
In this section we will show some examples and pictures of permutonestohedra. 

Let us start from \(A_2\). There is only one building set associated to this root system, since the maximal and the minimal building set coincide. There are six chambers and in every chamber the nestohedron is a segment.
Therefore in this case the permutonestohedron is a  dodecagon, and it is a  Kapranov's permutoassociahedron (see Figure \ref{provadodecagono}). 
 \begin{figure}[h]
 \center
\includegraphics[scale=0.22]{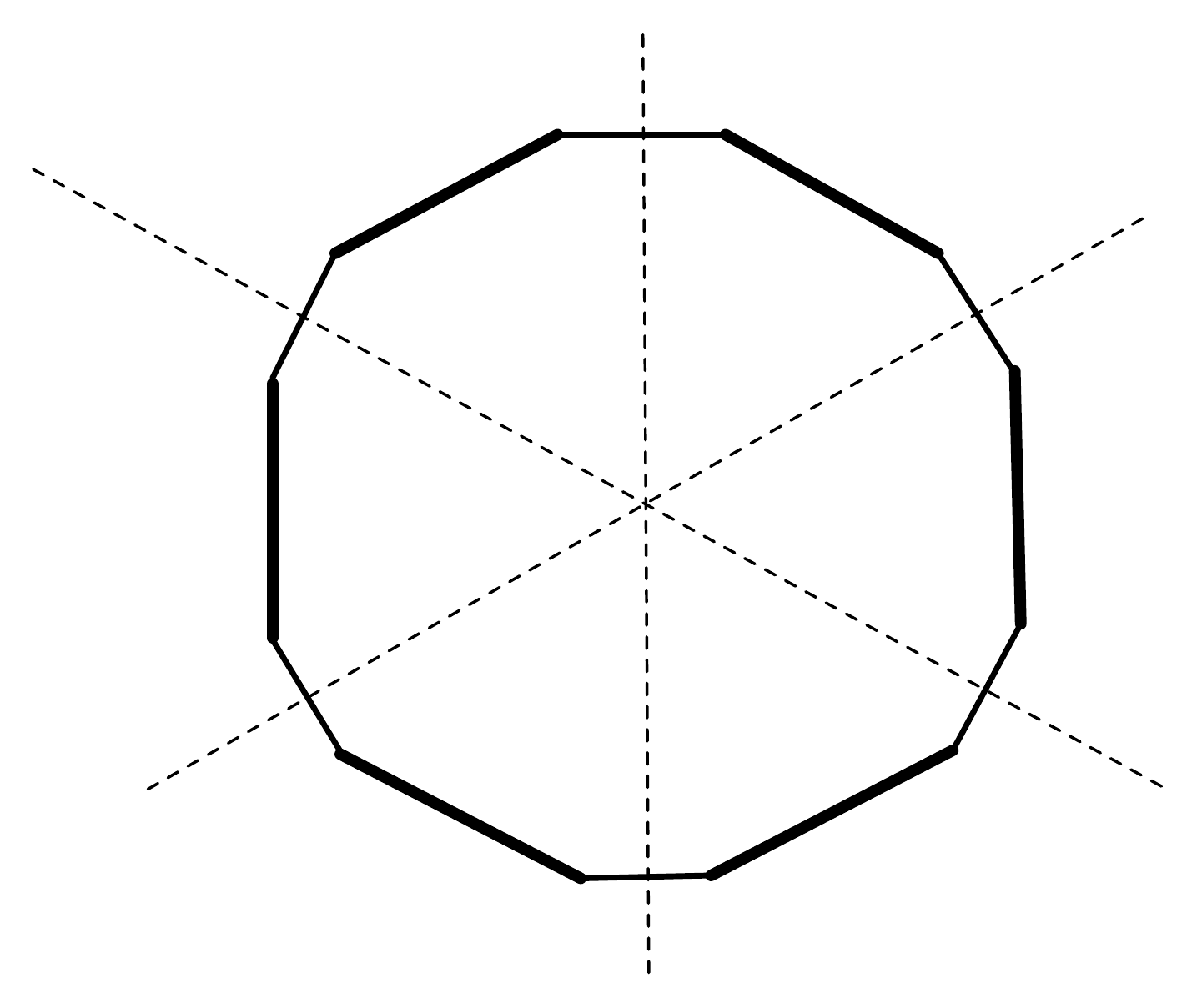}
\caption{The  permutonestohedron of type \(A_2\) (i.e. the two dimensional Kapranov's permutoassociahedron). The thick edges provide a linear realization of  \(CY_{\F_{A_2}}\):  once \(\epsilon_2\) is fixed, there is only one admissible value for \(\epsilon_1\) such that the dodecagon  is regular.}
\label{provadodecagono}
\end{figure}
It is not necessarily regular;  this depends on the choice of \(\epsilon_1,\epsilon_2\): once \(\epsilon_2\) is fixed, there is only one admissible value for \(\epsilon_1\) such that \(P_{\F_{A_2}}(A_2)\) is regular. If it is not regular  its edges have two different lengths and its automorphism group  coincides with  \(Aut \ A_2 \cong S_3\rtimes \Z_2\); if it is regular its automorphism group, which is the full dihedral group with 24 elements,   strictly contains \(Aut \ A_2\).

In the \(A_3\) case there are two distinct \(S_4\) invariant building sets: the building set of the irreducibles \(\F_{A_3}\) and the maximal building set. A picture of the corresponding permutonestohedra, which are a Kapranov's permutoassociahedron (\(P_{\F_{A_3}}(A_3)\)) and  a `permutopermutohedron', is in Figure \ref{permutonestoedria3bis}. It is easy to show  that for every choice of a suitable list \(\epsilon_1<\epsilon_2<\epsilon_3=a\) their automorphism group coincides with \(Aut \ A_3 \cong S_4\rtimes \Z_2\).  In the minimal case the nestohedra that  lie inside the chambers are pentagons (i.e., the two dimensional Stasheff's associahedra) while in the maximal case they are hexagons (i.e. the two dimensional permutohedra).

 \begin{figure}[h]
 \center
\includegraphics[scale=0.28]{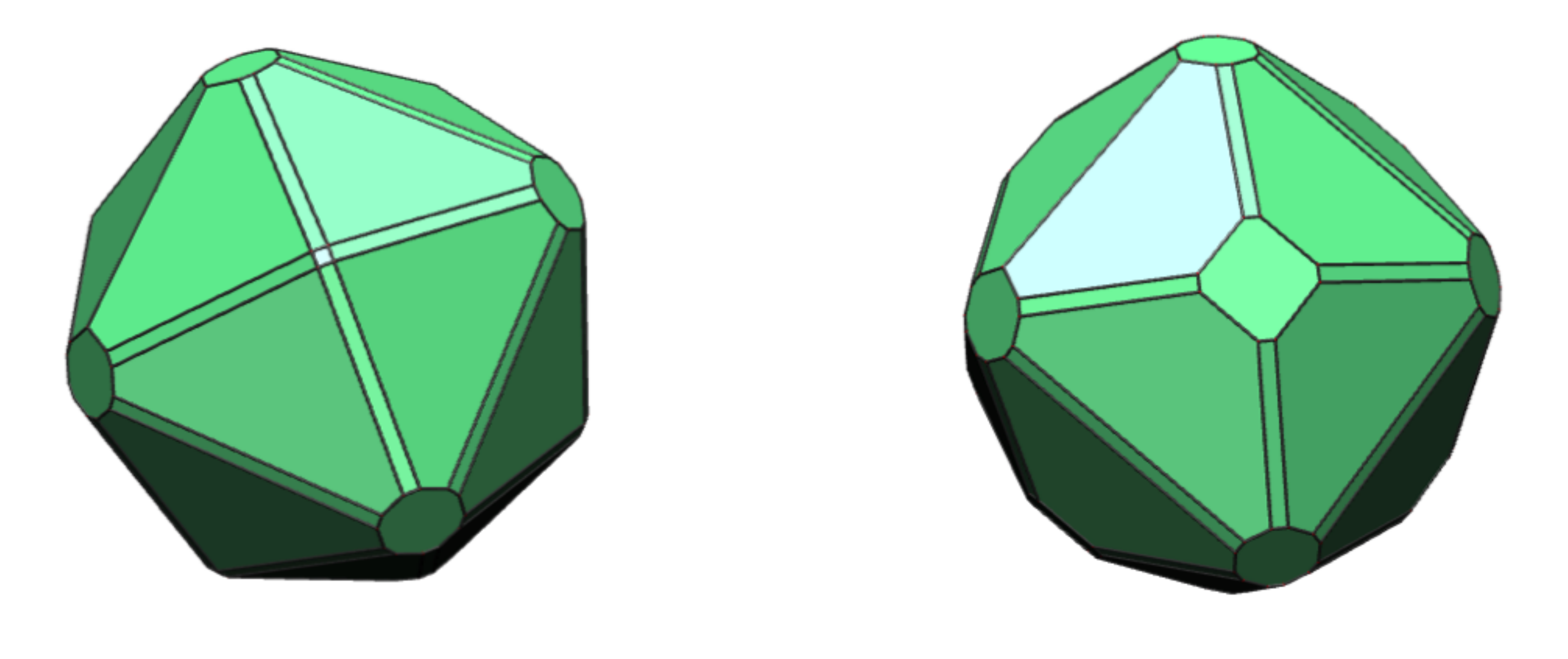}
\caption{The minimal (on the left)  and the maximal permutonestohedron of type \(A_3\). In accordance with Corollary \ref{cor:nonsemplice}, the maximal permutonestohedron is a simple polytope, while the minimal one is not simple. }
\label{permutonestoedria3bis}
\end{figure}

In the \(B_2\) case there is only one building set, and the corresponding permutonestohedron is a polygon with 16 edges. It is not regular and, depending on the choice of \(\epsilon_1\), its edges can have  two or  three different lengths;   its automorphism group is \(W_{B_2}\cong \Z_2^2\rtimes S_2\), i.e. the dihedral group with eight elements. 

In the \(B_3\) case we have two \(W_{B_3}\)-invariant building sets.  The corresponding permutonestohedra appear  in Figure \ref{permutonestoedrib3}. As in the \(A_3\) case, in the minimal case the nestohedra inside the chambers are pentagons while in the maximal case they are hexagons. 

The automorphism group   of these permutonestohedra is  \(W_{B_3}\cong \Z_2^3\rtimes S_3\) as it  follows for instance  from Corollary  \ref{criterioaut}.

 \begin{figure}[h]
 \center
\includegraphics[scale=0.41]{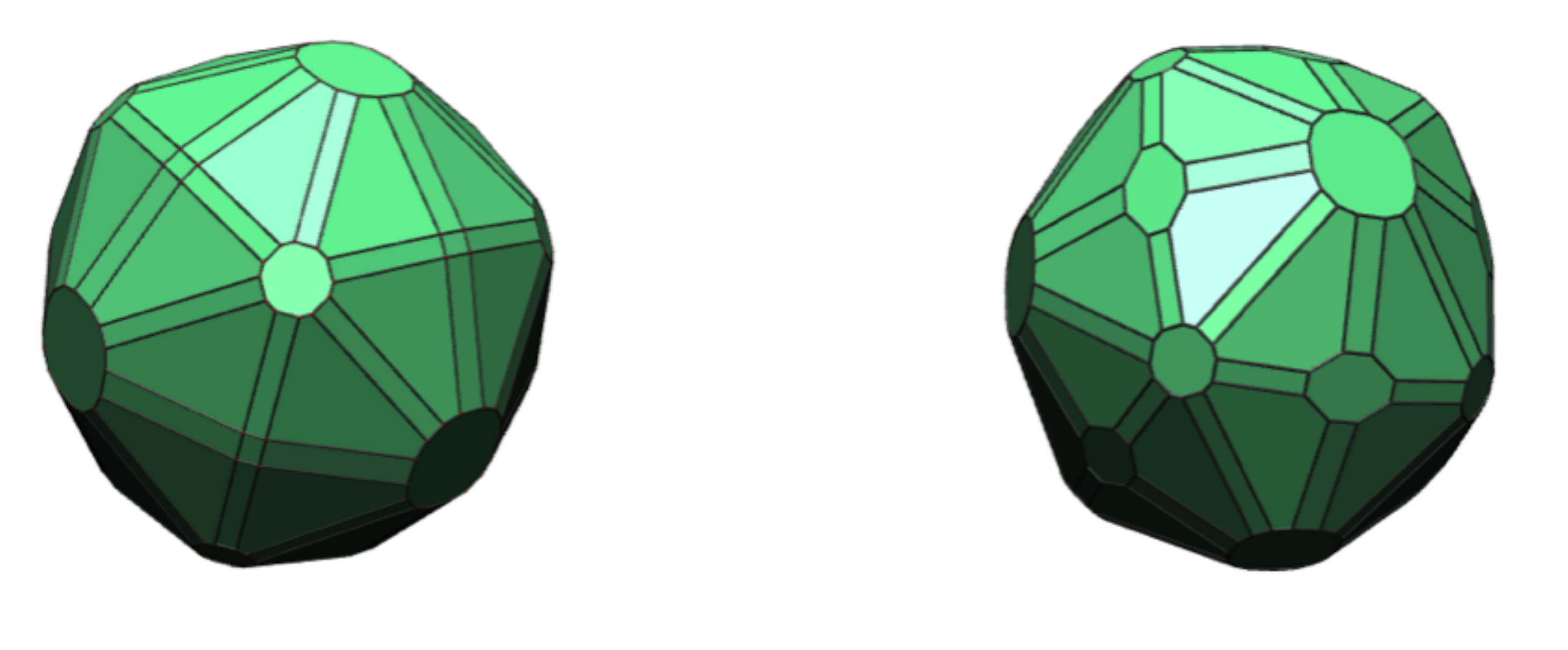}
\caption{The minimal (on the left)  and the maximal permutonestohedron  of type  \(B_3\).}
\label{permutonestoedrib3}
\end{figure}

Let us now consider the boolean arrangements \(Bo(n)\), i.e., the arrangements associated with  the root systems \(A_1^n\). The nestohedra \(P_{\G_{fund}}\), as \(\G\) varies among all the \(W=\Z_2^n\) - invariant building sets containing \(V\),  are all the nestohedra  in the `interval simplex-permutohedron' (see \cite{petric2}). 

In the \(A_1\times A_1\)  case there is only one possible building set which contains \(V\) and the corresponding permutonestohedron is an octagon. It may  be regular, depending on the choice of \(\epsilon_1\). If it is not regular its edges  have two different lengths and its automorphism group coincides with \(Aut \ A_1^2\cong \Z_2^2\rtimes S_2\) (\(\cong W_{B_2}\)).

In Figure \ref{boolean3max} there are two pictures  of the maximal permutonestohedron of type  \(A_1\times A_1 \times A_1\).
 Its automorphism group coincides with \(Aut \ A_1^3\cong \Z_2^3\rtimes S_3\) (\(\cong W_{B_3}\)).
We recall that the real De Concini-Procesi model associated with the maximal building set of the boolean arrangement \(Bo(n)\) is isomorphic to the toric variety of type \(A_{n-1}\) (see Procesi \cite{procesitoric}, Henderson \cite{HenPisa}).

 \begin{figure}[h]
 \center
\includegraphics[scale=0.43]{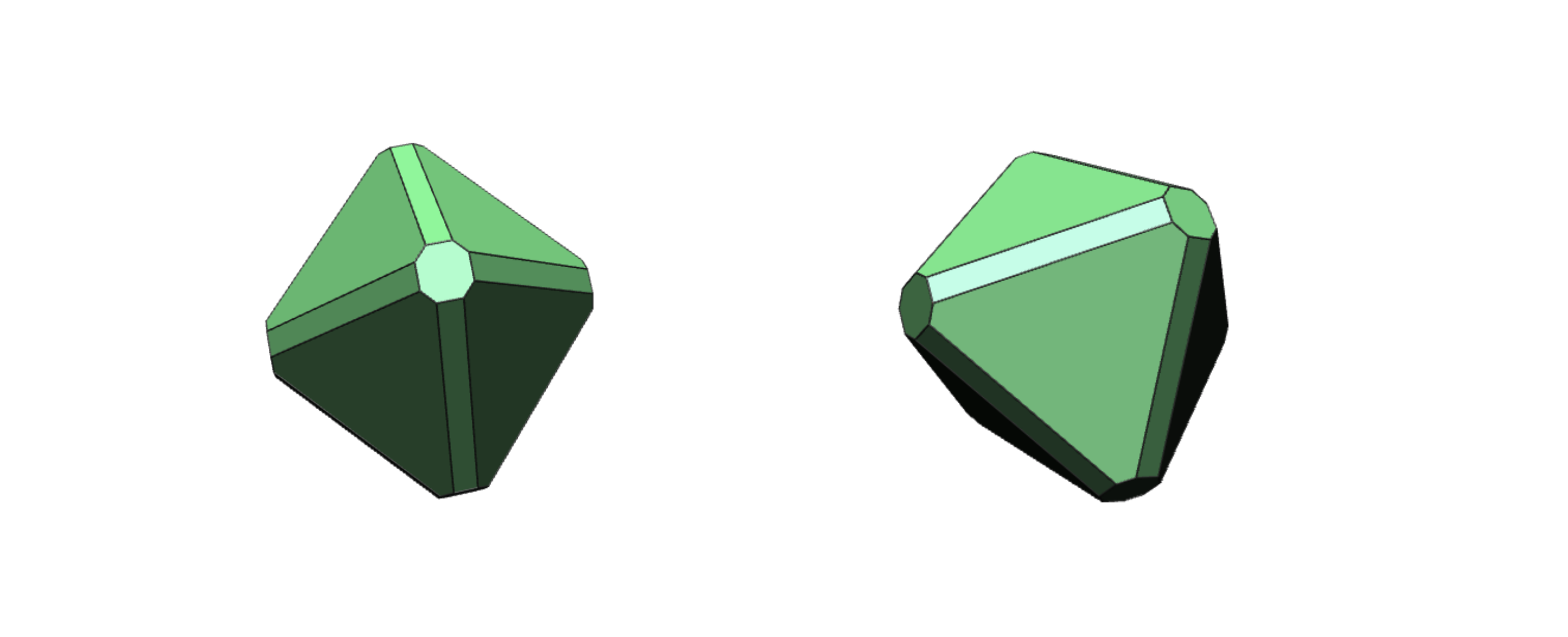}
\caption{Two pictures of the maximal permutonestohedron in the three dimensional boolean case (\(A_1\times A_1\times A_1\)).}
\label{boolean3max}
\end{figure}

  In the case of the root system \(A_4\)  there is  one \(S_5\)-invariant building set which strictly contains the minimal one and is different from the  maximal one.  Therefore  there is a permutonestohedron which is intermediate between the minimal and the maximal one.
 For any irreducible root system of dimension  \(n\geq 4\) intermediate building sets (i.e., not minimal or maximal) appear: in \cite{GaiffiServenti2} all the \(W\)-invariant  building sets \(\G\) of type \(A_n,B_n, C_n, D_n\) have been  classified.  

We observe that, if \(\Phi=A_n,B_n,C_n\) and we consider the minimal building set,  the corresponding permutonestohedron is   the convex hull  of \(|W|\) (resp. \(n!\), \(2^nn!\), \(2^nn!\)) \((n-1)\)-dimensional Stasheff 's associahedra. In  the \(A_n\) case this minimal permutonestohedron is a Kapranov's permutoassociahedron, and therefore it is combinatorially equivalent to  Reiner and Ziegler's  Coxeter  associahedron of type \(A_n\) (see \cite{reinerziegler}).  In the \(B_n\) and \(C_n\) case it is easy to check that the minimal permutoassociahedron is different from the corresponding Coxeter associahedron, since the latter is the convex hull of \(|W|\)   Stasheff's associahedra  whose dimension is \((n-2)\), not \((n-1)\). For instance, the Coxeter associahedron of type \(B_2\) is an octagon, while the unique  permutonestohedron of type \(B_2\) is   a polygon with 16 edges, as we observed before.

Furthermore, we remark  that if  \(\Phi=D_n\) the minimal permutonestohedon is not the convex hull of some Stasheff's associahedra: it is the convex hull of \(2^{n-1}n!\)  
graph associahedra of type \(D_n\) (for a description of these polytopes see for instance  \cite{carrdevadoss}  and  Section 8 of \cite{postnikov}). 

For any root system \(\Phi\), the  maximal permutonestohedron is the convex hull of \(|W|\) permutohedra.

\subsection{Counting faces}
\label{subsec:counting}

As an example of face counting on permutonestohedra, in this section we will compute the \(f\)-vectors of the minimal and  maximal permutonestohedra associated to the root system \(A_n\) (resp. \(P_{\F_{A_n}}(A_n)\) and \(P_{{\mathcal C}_{A_n}}(A_n)\)). These computations are  variations of the well known computations of the \(f\)-vectors of the Stasheff's associahedron and of the permutohedron.  
We first  need to introduce some notation.

For any positive integer \(n\) let us  denote by \(\Lambda_n\) the set of the partitions of \(n\) and, if \(\lambda\in \Lambda_n\) we denote by \(l(\lambda)\) the number of parts of \(\lambda\). Therefore we can write \(\lambda=(\lambda_1,\lambda_2,\ldots, \lambda_{l(\lambda)})\). A partition \(\lambda\in \Lambda_n\) can also be written as \(\lambda=\prod_{1\leq i\leq n}i^{m_i}\): this means that in \(\lambda\) the number \(i\) appears \(m_i\) times. We then put \[w(\lambda)=\frac{l(\lambda)!}{\prod_{1\leq i\leq n}{m_i!}}\]

 There is a bijective correspondence between the elements of $\F_{A_{n-1}}$ and the subsets of $\{1,\cdots,n\}$ of cardinality at
least two:  if   \(A^\perp\) is the subspace  described by the equation  \(x_{i_1}=x_{i_2}= \cdots =x_{i_k}\) then we represent \(A\in \F_{A_{n-1}}\) by the set \(\{i_1,i_2,\ldots, i_k\}\).
In an analogous way we can establish a bijective correspondence between the   elements of the maximal building set ${\mathcal C}_{A_{n-1}}$  and  the unorderd partitions of the set $\{1,\cdots,n\}$ in which at least one part has more than one element:
for instance, \(\{1,5,6\}\{2,3\}\{4\}\{7,8\}\) represents the  subspace  \(B\) in ${\mathcal C}_{A_{7}}$ of  dimension 4  such that \(B^\perp\)  is described by the system of equations \(x_{1}=x_{5}= x_{6}\), \(x_{2}=x_{3}\) and \(x_{7}=x_{8}\).

Now, to each unordered partition of $\{1,\cdots,n\}$ we can associate, considering the cardinalities of its parts, a partition in \(\Lambda_n\). Therefore we can associate a partition in \(\Lambda_n\) to every subspace in ${\mathcal C}_{A_{n-1}}$. We will say that a subspace in ${\mathcal C}_{A_{n-1}}$  has the form \(\lambda\in \Lambda_n \) if its associated partition is \(\lambda\). 
For instance, the subspace \(\{1,5,6\}\{2,3\}\{4\}\{7,8\}\)  in ${\mathcal C}_{A_{7}}$ has the form \((3,2,2,1)\).

\begin{teo}
\label{facceminimoan}
For every \(0\leq k\leq n-2\) the number of faces of codimension \(k+1\) of the  minimal permutonestohedron \(P_{{\mathcal F}_{A_{n-1}}}(A_{n-1})\) is
\[\sum_{\lambda\in \Lambda_n, \ l(\lambda)\geq2+k}w(\lambda)\frac{n!}{\lambda_1!\lambda_2!\cdots \lambda_{l(\lambda)}!}\left[  \frac{1}{k+1}\binom{l(\lambda)-2}{k} \binom{l(\lambda)+k}{k} \right]\]

\end{teo}

\begin{proof}
Let us denote by  \(\G\), for brevity,  the minimal  building set \(\F_{A_{n-1}}\).
As a first step we consider the faces represented in the poset by the pais \((\sigma H, \MS)\) such that the sum of the minimal subspaces of \(\MS\) which are non trivially labelled has the form \(\lambda\) for a fixed \(\lambda \in \Lambda_n\).  It is immediate to check that there are \(w(\lambda)\)  choices of a set of  minimal subspaces with this property in \(\G_{fund}\).

We then notice that, if \(\MS'\) is the complement in   \(\MS\) to the subset given by \(V\) and by the  non trivially labelled minimal subspaces, the codimension of the face represented by \((\sigma H, \MS)\) is equal to \(|\MS'|+1\). So, once the set of minimal non trivially labelled subspaces  is fixed (with the form \(\lambda\)), we  have to count in how many ways it is possible to complete this set of subspaces to a nested set \(\MS\) adding \(V\) and \(k\) other subspaces in \(\G_{fund}\).

This is equivalent to finding the number of distinct  parenthesizations with \(k\) couples of parentheses of an ordered list of \(l(\lambda)\) distinct elements. This number was computed by Cayley in \cite{Cayley}  and is equal to 
\[\frac{1}{k+1}\binom{l(\lambda)-2}{k} \binom{l(\lambda)+k}{k}  \]
To finish our proof it if sufficient to observe that the number of faces represented by the pairs \((\sigma H, \MS)\), once \(\MS\) and \(H\) are fixed, is equal to the index of \(H\) in \(W\), which in our case is \(\frac{n!}{\lambda_1!\lambda_2!\cdots \lambda_{l(\lambda)}!}\).

\end{proof}
\begin{rem}
In the case of the  faces of codimension \(n-1\), i.e. the vertices, the formula of Theorem \ref{facceminimoan}  specializes, as expected,  to \(C_{n-1}n!  \) where \(C_{n-1}\) is the Catalan number \(\frac{1}{n}\binom{2n-2}{n-1}\). 
\end{rem}
\begin{teo}
\label{faccemassimoan}
For every \(0\leq k\leq n-2\) the number of faces of codimension \(k+1\) of the  maximal permutonestohedron \(P_{{\mathcal C}_{A_{n-1}}}(A_{n-1})\) is
\[\sum_{\lambda\in \Lambda_n, \ l(\lambda)\geq2+k}w(\lambda)\frac{n!}{\lambda_1!\lambda_2!\cdots \lambda_{l(\lambda)}!}\left[ \sum_{1<j_1<\cdots<j_k<l(\lambda)=j_{k+1}} \prod_{t=1}^{k}\binom{j_{i+1}-1}{j_i-1} \right]\]

\end{teo}
\begin{proof}
Let us denote by  \(\G\), for brevity,  the maximal building set \({\mathcal C}_{A_{n-1}}\).
As in the first part of the proof of the preceding theorem, we observe that, given \(\lambda \in \Lambda_n\)  there are \(w(\lambda)\) subspaces in \(\G_{fund}\) whose form is \(\lambda\).
So, once \(\lambda\) if fixed,  we have \(w(\lambda)\) ways to choose a minimal non-trivially labelled subspace  whose form is \(\lambda\).  Now we  have to count in how many ways it is possible to complete this subspace to a nested set \(\MS\) adding \(V\) and \(k\) other subspaces in \(\G_{fund}\). Since the nested sets in this maximal case are linearly ordered by inclusion, we can do this in \(k\) steps.
At the first step we have to add a subspace which contains the minimal one, which is equivalent to split into  \(j_k\) parts (with \(1<j_k<l(\lambda)\)) an ordered list of \(l(\lambda)\) distinct elements.
Then at the second step we have to add a subspace which contains the one added at the first step, which is equivalent to split into  \(j_{k-1}\) (with \(1<j_{k-1}<j_k\)) parts  an ordered list of \(j_{k}\) distinct elements.
This process is successful if we manage to complete    \(k\) steps of this kind; this happens in
\[ \sum_{1<j_1<\cdots<j_k<l(\lambda)=j_{k+1}} \prod_{t=1}^{k}\binom{j_{i+1}-1}{j_i-1} \]
 different ways, where the sum ranges over all the possible lists of \(k\) numbers \(j_i\) that  satisfy the required inequalities.
 The claim then follows, as in the preceding theorem, by counting the index of the subgroup given by the non trivial label.

\end{proof}

\begin{rem}
The formula of Theorem \ref{faccemassimoan} in particular claims that  the vertices of  \(P_{{\mathcal C}_{A_{n-1}}}(A_{n-1})\) are \((n-1)! n!\), as expected from a `permutopermutohedron'. 
\end{rem}

\addcontentsline{toc}{section}{References}
\bibliographystyle{acm}
\bibliography{Bibliogpre} 
\end{document}